\documentclass[11pt,reqno]{amsart}

\usepackage{enumerate,url,amssymb,  mathrsfs}
\usepackage{graphicx}
\pdfoutput=1

\newtheorem{theorem}{Theorem}[section]
\newtheorem{lemma}[theorem]{Lemma}
\newtheorem*{lemma*}{Lemma}
\newtheorem{proposition}[theorem]{Proposition}
\newtheorem{corollary}[theorem]{Corollary}

\theoremstyle{definition}

\theoremstyle{remark}
\newtheorem{remark}[theorem]{Remark}

\numberwithin{equation}{section}

\newcommand{\refeq}[1]{(\ref{#1})}

\newcommand{\C}{\mathbb{C}}

\newcommand{\DD}{\mathbb{D}}

\newcommand{\R}{\mathbb{R}}
\newcommand{\Z}{\mathbb{Z}}

\newcommand{\Hpl}{\mathbb{H}}

\DeclareMathOperator{\diam}{diam}

\DeclareMathOperator{\re}{Re}
\DeclareMathOperator{\im}{Im}

\def\XXint#1#2#3{{\setbox0=\hbox{$#1{#2#3}{\int}$}
\vcenter{\hbox{$#2#3$}}\kern-.5\wd0}}

\def\ge{\geqslant}
\begin{document}
\baselineskip6mm
\vskip0.4cm
\title[Quasiconformal Multifractal Spectra]{Bilipschitz and Quasiconformal Rotation, Stretching and Multifractal Spectra}

\author[K. Astala]{Kari Astala}
\address{Department of Mathematics and Statistics, University of Helsinki, 
         P.O. Box 68, FIN-00014, Helsinki, Finland}
\email{kari.astala@helsinki.fi}

\author[T. Iwaniec]{Tadeusz Iwaniec}
\address{Department of Mathematics, Syracuse University, Syracuse,
NY 13244, USA and Department of Mathematics and Statistics,
University of Helsinki, Finland}
\email{tiwaniec@syr.edu}

\author[I. Prause]{Istv\'an Prause}
\address{Department of Mathematics and Statistics, University of Helsinki,
         P.O. Box 68, FIN-00014, Helsinki, Finland}
\email{istvan.prause@helsinki.fi}

\author[E. Saksman]{Eero Saksman}
\address{Department of Mathematics and Statistics, University of Helsinki, 
         P.O. Box 68, FIN-00014, Helsinki, Finland}
\email{eero.saksman@helsinki.fi}

\thanks{
K.A.~was supported by the Academy of Finland (SA) grant 75166001 and 1134757. T.I.~was supported by NSF grant DMS-1301558 and SA 1128331. I.P. was supported by SA 138896 and SA 1266182.
E.S.~was supported by SA 113826, SA 118765 and by the Finnish CoE in Analysis and Dynamics Research.}



\maketitle


\begin{abstract} We establish sharp bounds for simultaneous local rotation and H\"older-distortion  of planar quasiconformal maps. In addition, we give sharp estimates for the corresponding joint quasiconformal multifractal spectrum, based on new estimates for Burkholder functionals with complex parameters. As a consequence, we obtain optimal rotation estimates  also for bi-Lipschitz maps.
\end{abstract}

\section{Introduction}\label{se:introduction}

A deformation $f \colon \R^2 \to \R^2$ is called $L$-{\it bilipschitz} if it distorts Euclidean distances by at most a fixed factor $L \geqslant 1$,
\[ \frac{1}{L} \, | x -y | \leqslant |f(x) - f(y) | \leqslant L \, |x-y| \qquad \mbox{for } x, y \in \R^2.
\] 
Such a map changes length insignificantly, nevertheless it may change local geometry  by creating (logarithmic) spirals out of line segments. As a simple model example consider the logarithmic spiral map
\begin{equation}
\label{model1}
f(z) =  z \,|z|^{i\gamma}, \qquad z \in \C = \R^2, \quad \gamma \in \R\setminus \{0\}.
\end{equation}

On the other hand, the constant $L$ imposes constraints on the speed of spiralling:  for  any $L$-bilipschitz map, see Proposition \ref{thm:pointwise} below,  the infinitesimal rate of rotation at a point $z \in \R^2$,
\begin{equation}
\label{2}
\gamma(z) = \gamma_{f}(z) = \limsup_{t \to 0} 
\frac{ \arg\bigl[ f (z+t) - f (z)\bigr] }{ \log \, t } ,
\end{equation}
satisfies  $|\gamma|  \leqslant L - \frac{1}{L}$ with equality for \eqref{model1} at $z=0$.

Note that a map can have non-trivial (i.e. $\gamma \neq 0$) rotation only at points of non-differentiability, hence for bilipschitz maps only in  a set of measure zero. This leads one to study  the \emph{rotational multifractal spectrum} of these mappings, that is, to ask what is the maximal Hausdorff dimension of a set $E_\gamma$  where 
an $L$-bilipschitz deformation   rotates every point  with a  given rate $\gamma$.

The following result gives a complete answer, describing the universal bounds for the class of bilipschitz mappings. Indeed, these interpolate linearly between the bounds valid pointwise and those valid almost everywhere.

\begin{theorem} 
\label{1.1} Suppose $f: \R^2 \to \R^2$ is $L$-bilipschitz and $\gamma$ is a real number such that $|\gamma|  \leqslant L - \frac{1}{L}$. Then 
\begin{equation}
\label{lipdim}
\dim_{\mathcal H}\{ z \in \R^2: \gamma_f(z) = \gamma \} \;  \leqslant 2  \,-\, \frac{2L}{L^2-1}|\gamma|.
\end{equation}
Moreover, for every such $\gamma$ 
there exists an $L$-bilipschitz map $f: \R^2 \to \R^2$ for which the equality holds in \eqref{lipdim}.
\end{theorem} 

A key observation towards establishing rotational bounds, such as Theorem \ref{1.1}, is that it is necessary to control the higher integrability of  {\it complex powers} of the derivative $f_z=\partial f/\partial z$. Such an integrability is most naturally studied through a holomorphic flow of the map $f$. 

However, holomorphic flows  do not keep the maps bilipschitz.  Any orientation preserving $L$-bilipschitz   $f:\R^2 \to \R^2$ solves the Beltrami equation
\begin{equation}
\label{lipeq}
 f_{\bar{z}} = \mu(z) f_z, \qquad \mbox{ with} \quad |\mu(z)| \leq \frac{L^2- 1}{L^2+1} \quad \mbox{a.e. } x \in \R^2.
\end{equation}
Replacing $\mu$ by $ \mu_\lambda$, a coefficient depending holomorphically on a complex parameter $\lambda$, the equations provide us with homeomorphic solutions which vary holomorphically with $\lambda$   -- holomorphic motions of Ma{\~n}{\'e}, Sad and Sullivan \cite{MSS}. 
However, in general these homeomorphisms are only quasisymmetric (see \eqref{qsymmetry} below) even if  the initial map in  \eqref{lipeq} is bilipschitz.

Surprisingly, though we are forced to move outside the bilipschitz world, this approach leads to a complete description of the multifractal rotation spectra.

\medskip

The correct setup for integral estimates of complex powers of $f_z$ is within the solutions to an arbitrary Beltrami equation \eqref{lipeq},  thus in terms of quasiconformal mappings. Here the  result takes the following form (for the appropriate concepts  see Section \ref{se:pre}).

\begin{theorem}
\label{complexintegrability1}
Suppose  $f$ is a $K$-quasiconformal map on a domain $\Omega\subset\C.$ Then for any exponent $\beta \in \C$ in the critical ellipse
\begin{equation} \label{aito37}
 |\beta| + |\beta -2| <  2\cdot \frac{K+1}{K-1}
 \end{equation}
we have 
\begin{equation*}
\left| f_{z}^\beta \right| \in L^1_{loc}(\Omega).
\end{equation*}
\end{theorem}

This result is sharp in a strong sense; it fails for any complex exponent $\beta \,$ on the boundary of the critical ellipse as well as outside, see Section \ref{se:burkholder}.

Theorem \ref{complexintegrability1} includes a number of interesting special cases.  Taking $\beta$ real-valued we recover the optimal higher integrability of the gradient of a quasiconformal mapping \cite{As}. Other values of $\beta$  lead to new phenomena: among them is the optimal exponential integrability of the argument of $f_z$.

\begin{corollary}\label{cor:argumentintegrability-intro} Suppose  $f$ is a $K$-quasiconformal map on a domain $\Omega\subset\C$. Then
$$
e^{b |\arg f_z|}\in L^1_{loc}\quad {\rm for \; all}\;\;{\rm positive }\;\;   b< \frac{4K}{K^2-1}.
$$
\end{corollary}

Again the integrability fails for some $K$-quasiconformal mapping in case  $b= \frac{4K}{K^2-1}$.
Moreover, we have similar  precise bounds on exponential integrability  
for bilipschitz mappings as well, see Theorem \ref{thm:exparg} below. 

\medskip
To understand the mechanisms of how the integrability breaks down at the borderline case in Theorem \ref{complexintegrability1}, we analyze the situation in a weighted setting.
This reveals connections with Burkholder functionals and raises new questions regarding quasiconvexity. The novelty here lies in creating a  family of Burkholder-type functionals depending on a complex parameter,  for which we establish partial quasiconcavity, see Theorem \ref{main}. In the course of doing so we bring into play Lebesgue spaces with complex exponents to advance in this setting the interpolation technique of \cite{AIPS}.
For details see Section \ref{se:burkholder}. 

The above approach allows a natural extension 
 to  multifractal properties  of quasiconformal mappings. Even further,  in view of \cite{MSS}, \cite{Slod} one may consider  general holomorphic motions of subsets $E \subset \C$, that is,   maps $\Psi: \DD \times E \to \C$  injective in the $z$-variable and  holomorphic in  the $\lambda$-variable, with $\Psi(0,z) \equiv z$ at the ``time'' $\lambda = 0$.
Now, however, one needs to take into account also the stretching in a manner to be discussed  in  Section \ref{infit}. We arrive at  general bounds for the \emph{joint} rotational and stretching multifractal spectra.

\begin{theorem} \label{holodim}
Suppose  $\Psi: \DD \times E \to \C$ is a holomorphic motion of a set $E \subset \C$ and  that $\alpha > 0$ and $\gamma \in \R$ are given.  

If $\lambda \in \DD$, assume that at every  
point $z \in E$ we have  scales $r_j \to 0$ along which $\Psi_\lambda(z)=\Psi(\lambda,z)$  stretches with exponent $\alpha$,
$$ \lim_{j\to\infty}\frac{\log |\Psi_\lambda(z+r_j)-\Psi_\lambda(z)|}{\log r_j} =  \alpha, \qquad z \in E,$$ 
 and simultaneously rotates with rate $\gamma$, 
$$ \lim_{j\to\infty}\frac{\arg (\Psi_\lambda(z+r_j)-\Psi_\lambda(z))}{\log |\Psi_\lambda(z+r_j)-\Psi_\lambda(z)|}= \gamma,  \qquad z \in E.$$
Then 
\begin{equation}
\label{dimbound2}
\dim(E)\;  \leqslant \;   1+\alpha \,-\, \frac{1}{|\lambda|} \sqrt{(1-\alpha)^2+(1-|\lambda|^2)\alpha^2 \gamma^2}.
\end{equation}
\end{theorem}
\smallskip

Moreover, for each $\lambda \in \DD$, $\alpha > 0$ and $ \gamma \in \R$ such that the right hand side of  \eqref{dimbound2} is nonnegative, there exists a set $E \subset \C$ and its holomorphic motion for which we have equality in \eqref{dimbound2}.

Holomorphic motions and the study of their geometric properties arise naturally in various questions in  complex dynamics. It is clear that detailed combinatorial or topological information about specific dynamical systems, combined with the methods of Theorem \ref{holodim}, will improve the bounds above.

The proofs of the statements in this introduction can be found in the text as follows.
Theorem \ref{complexintegrability1} and Corollary \ref{cor:argumentintegrability-intro} are proved in Section \ref{se:burkholder}. Theorem \ref{holodim} and further quasiconformal multifractal spectra estimates are treated in Section \ref{se:multifractal}. Finally, Theorem \ref{1.1} is obtained by developing the bilipschitz theory in Section \ref{se:bi-Lipschitz}. 
In Section \ref{se:rotation} we build up the basic framework for discussing rotational phenomena through rigorous definitions of various branches of the logarithms involved. Section \ref{se:pre} contains prerequisites on bilipschitz and quasiconformal mappings.

\medskip
The above and further results in the present work as well as of other authors reflect a close interaction between rotational phenomena for bilipschitz mappings and stretching phenomena for quasiconformal mappings. This may be summarized as the \emph{dictionary} of Table \ref{table:qcvsbilip} between the two, see Section \ref{se:bi-Lipschitz} for a discussion.

{
\setlength{\tabcolsep}{10pt}
\renewcommand{\arraystretch}{1.2}

\bigskip
\begin{table}[htbp]
\begin{center}
\begin{tabular}{|c|c|}
\hline
{\bfseries $K$-quasiconformal stretching}  & {\bfseries $L$-bilipschitz rotation} \\ \hline \hline
Gr\"otzsch problem & John's problem \\ \hline
radial stretching & logarithmic spiral map\\ 
$z |z|^{\alpha-1}$ & $z |z|^{i \gamma}$ \\ \hline
H\"older exponent & rate of spiralling \\
$1/K \leqslant \alpha \leqslant K$ & $|\gamma| \leqslant L-1/L$ \\ \hline
$\log J(z,f) \in BMO$ & $\arg f_z \in BMO$ \\ \hline
higher integrability & exponential integrability \\
$f \in W^{1,p}_{loc}, \, p <\frac{2K}{K-1}$ & $\exp(b |\arg f_z|) \in L^1_{loc}, \, b <\frac{2L}{L^2-1}$ \\ \hline
multifractal spectrum & multifractal spectrum \\
 $\dim_H \{z : \alpha(z)=\alpha\} \leqslant 1+\alpha-\frac{|1-\alpha|}{k}$ & $\dim_H \{w: \gamma(w)=\gamma\} \leqslant 2-\frac{2L}{L^2-1}\, |\gamma|$ \\ \hline
 factoring the radial stretch map  & factoring the logarithmic spiral map \\
 along a geodesic & along a horocycle \\ \hline
\end{tabular}
\end{center}
\bigskip
\caption{Quasiconformal stretching versus Bilipschitz rotation}
\label{table:qcvsbilip}
\end{table}
}

\section{Prerequisites}\label{se:pre}

 Before considering 
the   rotational and  stretching properties of  bilipschitz and
quasiconformal mappings,  we briefly recall  the basic concepts underlying our study.

\subsection{Bilipschitz and Quasiconformal Maps}\label{bilipqc}

 By definition, in any dimension $n \geqslant 2$, $K$-quasiconformal mappings are orientation preserving homeomorphisms $f:\Omega \to \Omega'$ between domains $\Omega, \Omega' \subset \R^n$, contained  in the Sobolev class $\mathscr W^{1,n}_{loc}(\Omega)$, for which  the differential matrix   and its determinant are coupled in the distortion inequality,
\begin{equation}\label{distortion}
 |D\!f(x)|^n  \leqslant K\, \det D\!f(x)\;,\quad \textrm{where}\;\;\;|D\!f(x)|  = \max_{|\xi| =1} \; |D\!f(x) \xi|,
\end{equation}
for some  $K \geqslant1$.
From now on, we will consider  these mappings only in dimension $n=2$. 
It is clear that any orientation preserving $L$-bilipschitz mapping is $K$-quasi\-conformal with $K= L^2$.

In comparison, any quasiconformal mapping of the entire plane is {\it quasisymmetric}, i.e. satisfies the estimate
\begin{equation}
\label{qsymmetry}
\frac{|f(x) - f(z)|}{|f(y) - f(z)|} \leqslant \eta\left( \frac{|x-z|}{|y-z|} \right), \qquad x, y, z \in \C,
\end{equation}
for some continuous strictly increasing $\eta:\R_+ \to \R_+$ with $\eta(0)=0$. We will make frequent use of this geometric description. Conversely, any quasisymmetric mapping is $K$-quasiconformal with $K = \eta(1)$.

What makes quasiconformal mappings particularly flexible in dimension $n=2$ is that they satisfy the Beltrami equation
\begin{equation}\label{eq:beltrami2}
f_{\bar z} = \mu f_z\quad {\rm  with}\quad |\mu| \leqslant k < 1, \qquad k = \frac{K-1}{K+1}.
\end{equation} 
Conversely, any homeomorphic $ W^{1,2}_{loc}$-solution to \eqref{eq:beltrami2} is $K$-quasiconformal with  $K=(1+k)/(1-k).$ We refer to
\cite{AIMb} for definitions and basic facts on planar quasiconformal maps.

If the coefficient  $\mu$ is  compactly supported, the mapping $f$ is conformal near $\infty$ and we may use normalisation  $f(z) =  z+o(1)$  as $z\to\infty $. In this case we call $f$ the {\it principal solution} to \eqref{eq:beltrami2}; such a solution is uniquely determined by the  Beltrami coefficient $\mu(z)$ and it is a homeomorphism.
The Beltrami equation paves the way for embedding  the  principal solution $f$ into a holomorphic family of quasiconformal maps. One may simply  consider the flow of principal solutions $\{f^\lambda(z)\}_{\lambda \in \DD}$ to the equation
\begin{equation}
\label{eq:flow}
f_{\bar z} = \mu_\lambda f_z, \qquad \mu_\lambda(z) = \lambda \mu(z) / \| \mu \|_{\infty}, \quad f=f^\lambda.
\end{equation}
Similarly,  the global solutions to \eqref{eq:flow} normalized by $f(0)=0, \; f(1)=1$  depend holomorphically on $\lambda \in \DD$. More generally, within both normalizations we achieve  the holomorphic dependence as soon as $\mu_\lambda(z)$ varies holomorphically with the parameter $\lambda$. For details and further facts on holomorphic dependence see \cite[Section 5.7]{AIMb}.

\subsection{Choosing the logarithmic branches}\label{se:logarithmic branches}

In describing the different aspects of rotation under quasiconformal mappings we will need a convenient and systematic way to discuss the various branches of the logarithm of the difference $f(z)-f(w)$. The following simple observation  serves best our purposes.

\begin{proposition} \label{branch5}
Let $f:\Omega \to \mathbb R^2$ be a homeomorphism of a simply connected domain $\Omega \subset \mathbb R^2\,$.
 Assume that  $f$ is differentiable at a point $w_0 \in \Omega\,$ with positive Jacobian. Then the logarithmic expression
 \vspace{-.1cm}

 \begin{equation}
\label{logs4}
\qquad \; \log \frac{f(z)-f(w)}{z-w}\,, \qquad 
(z,w) \in  \Omega \times \Omega, \quad  z \neq w,
\end{equation}
 \smallskip
 \vspace{-.4cm}

\noindent admits a single-valued continuous branch.

Moreover, any continuous branch is uniquely determined by its value at any given pair $(z,w) \in \Omega \times \Omega$, $z\neq w$.
\end{proposition}
\begin{proof}
Since $f:\Omega \to \mathbb R^2\,$ is differentiable at $w_0 \in \Omega$ and $J(w_0,f) = |f_z(w_0)|^2 - |f_{\bar z}(w_0)|^2 > 0$, the function
\begin{equation}
\label{logprod}
z \mapsto \frac{f(z)-f(w_0)}{z-w_0} 
\end{equation}
is continuous and non-vanishing in $\Omega \setminus \{ w_0\}$. Its accumulation set at $z=w_0$ is contained in a closed disk $\bar B(a,r)$, centered at $ a = f_z(w_0)$ and of  radius $\,r = | f_{\bar z}(w_0)| \geqslant 0\,$.  By our assumptions on the Jacobian, $\,0\not\in B(a, R)$ as soon as $\,R >r$ is sufficiently close to $\,r\,$. If $\gamma$ is any loop in $\Omega\setminus\{ w_0 \}$, it is homotopic to a loop that lies arbitrarily close to $w_0$, i.e. to a loop whose image under $\,f\,$ lies in $\,B(a,R)\,$. Therefore \eqref{logprod} admits a single valued logarithm as a function of $z \in \Omega$.

In order to show that  $\; \log \frac{f(z)-f(w)}{z-w}\,$ admits a single-valued continuous branch in the whole domain  $U:=(\Omega \times \Omega) \setminus\{(w,w):  w \in \Omega\}$, it is  then enough to verify that any loop in $U$ is homotopic  to one that lies in the section
$U\cap \{ (z,w)\,:\, w=w_0 \}.$ Since $\Omega$ is simply connected, it is homeomorphic to the unit disc  $\DD$ and we just need to
consider the case $U=\DD \,$. Moreover, we may assume that $w_0=0.$ Given any loop $\,(\alpha,\beta):[0,1]\to \DD\times \DD\,$
that avoids the diagonal (that is, $\alpha(t) \neq \beta(t)\,\textnormal{for all}\, 0\leqslant t \leqslant 1\,$), the required homotopy is given by

\[
 (\alpha_s\,,\, \beta_s ) = \left(\,\frac{\alpha - s \beta}{1 - \,s \overline{\beta}\alpha }\,,\;\frac{(1 - s)\, \beta}{1 - \,s\, |\beta|^2} \right)\;\;,\;\;\;\;0\leqslant s \leqslant 1
\]
Indeed,  $\alpha_s(t) \neq \beta_s(t)\,$ for all $\, 0\leqslant t \leqslant 1\,$ and $\, 0\leqslant s \leqslant 1\,$.
 \end{proof}

\begin{remark}\label{rem:principal}
In most cases that we will encounter there is a natural choice for the branch of the logarithm \eqref{logs4}. For instance, if $f:\C \to \C$ is normalized by $f(0)=0, \; f(1)=1$ we will choose
$\log f(1)=\log \big((f(1)-f(0)/(1-0)\big) =0$. Or if $f:\C \to \C$ is a principal solution to \eqref{eq:beltrami2} with  $f(z)= z+o(1)$   as $z\to\infty$, then we consider the continuous branch with
\begin{equation}
\label{branch}
\log \frac{f(z)-f(w)}{z-w} \;  \to \;  0 \qquad \mbox{as } z\to\infty. 
\end{equation}
and call it the  {\it principal branch}.
\end{remark}

\begin{remark}\label{rem:flow}
Later on, we  also need to consider  {\it flows of homeomorphisms} $f^\lambda:\C\to\C$  that depend continuously on a complex parameter $\lambda\in\DD .$ Proposition \eqref{branch5} generalizes to this situation, and there is a single-valued continuous  branch of the logarithm
\smallskip
$$
 \log \frac{f^\lambda(z)-f^\lambda(w)}{z-w} \quad {\rm defined \; for }\quad 
(z,w) \in  \C \times \C, \quad  z \neq w,\quad {\rm and}\;\; \lambda\in\DD.
$$
In order to prove this one observes that any loop $\gamma=(\gamma_1,\gamma_2,\gamma_3)$ in the domain  $$\big((\C \times \C)\setminus \{(z,w)\, : \,z=w\}\big)\times \DD$$ is homotopic to a curve lying in 
$\big((\C \times \C)\setminus \{(z,w)\, : \,z=w\}\big)\times \{ 0\}$
by the trivial homotopy $(t,s)\mapsto (\gamma_1(t),\gamma_2(t),(1-s)\gamma_3(t)).$ This reduces one to the case that was already handled in the proof of Proposition \eqref{branch5}. Naturally the branch is uniquely determined by its value at any $(z,w,\lambda)$ with $z\not= w.$
   
\end{remark}

\section{Notions of rotation}\label{se:rotation}

We start with the pointwise notions, i.e. describe the extremal behavior and the optimal bounds for the rotation and stretching at a point $z_0 \in \C$. 

There are (at least)  three natural different and geometric ways to describe the rotational properties under a planar mapping:  the infinitesimal, the local and the global concepts. 
They have different and complementary geometric and analytic descriptions, but we shall see later that they are all intimately related. These relations will then form the basis of the quasiconformal and bilipschitz multifractal properties  studied in Sections \ref{se:multifractal} and \ref{se:bi-Lipschitz}.

We begin with the local point of view, where one fixes the argument at a chosen pair of points $z_0 \neq z_1 \in \C$, and then studies the behavior of the logarithmic cross ratio 
\[ \log \left(\frac{f(z) - f(z_0)}{f(z_1) - f(z_0)}\right), \qquad \mbox{ as } \;  z \to z_0, \quad z \neq z_0. \]
The scale invariance of this expression allows universal distortion bounds.

As a next step, the local bounds enable one to study  the geometric rate of  spiraling for  the  image of an infinitesimal line segment; for  precise definitions see  \eqref{model12} - \eqref{modelg} and Theorem \ref{localrotstret}.

Finally, for principal mappings; that is, when $f(z) = z+o(1)$   as $z\to\infty $, we have the third possibility of taking $z= \infty$ as the reference point. Here it is natural to travel from  $\infty$ to $z_0$  along a line segment and study  how much the image rotates  around $f(z_0)$.

\subsection{Rotation and Stretching: The local picture}\label{se:localrotation} 

By the classical  Mori's theorem, a $K$-quasi\-conformal mapping $f$ is locally H\"older continuous with exponent $1/K$. We look for a similar bound on the maximal change in argument. It turns out that a fruitful point of view is to measure twisting and stretching simultaneously, and this is one of the leading insights  in what follows. By scale invariance it makes no difference to take $z_0 = 0$ and $z_1 = 1$ and assume $f(0) =0, f(1) = 1$.


\begin{theorem} \label{ptwise}
Consider a $K$-quasiconformal mapping $f:\C \to \C$ normalized by $f(0)=0$ and $f(1) = 1$. Then for any 
$0  < r < 1$,
\begin{equation}
\label{logs}
 \left| \, \log f(r)\, - \,\frac12 \left(K+\frac{1}{K} \right) \,  \log r \right| \leqslant \frac12 \left(K-\frac{1}{K} \right)\, \log \frac{1}{r} + c(K),
 \end{equation}
 \vspace{.002cm}

\noindent where the constant $c(K)$ is independent of $\;r$. Here the continuous branch of the logarithm $\log f(r)$ is determined by $\log f(1) = 0$.

Furthermore, for any $0 < r < 1$ and for any number $\tau$ on the circle 

\begin{equation}
\label{etas}
\left| \, \tau -  \,\frac12 \left(K+\frac{1}{K} \right) \right| = \frac12 \left(K-\frac{1}{K} \right)
 \end{equation}
 \vspace{.002cm}

\noindent  the equality in \eqref{logs} holds with $c(K) = 0$ for the  
$K$-quasiconformal mapping 
\begin{equation}
\label{eq:gammaalpha}
f_\tau(z) = \frac{z}{|z|} |z|^{\tau}, \qquad \tau = \alpha(1+i\gamma). 
 \end{equation}
\end{theorem}
\smallskip
  

For an illustration of the parameter disk \eqref{etas}, see Figure \ref{fig:hypdisk} presented after Theorem \ref{localrotstret}.

\medskip

Before embarking to the proof, let us recall some of the earlier results on quasiconformal rotation.
Gutlyanski{\u\i} \;  and Martio \cite{GM} answer a problem of John \cite{Joh} on bilipschitz mappings by determining  the maximal rotation  for $K$-quasiconformal maps that fix given annuli, keep one boundary circle fixed and rotate the other boundary circle. The extremal in this situation is given by the map \refeq{eq:gammaalpha} where $\alpha=1$, corresponding  to a  pure rotation. 
Balogh, F\"assler and Platis \cite {BFP} extend this result to annuli with different modulus, and the extremal is of the general form \refeq{eq:gammaalpha}. Both of these works actually consider a fairly general class of maps of finite distortion. However, they only consider mappings between round annuli. Our result relaxes on this hypothesis,  and we simply ask for the maximal combined rotation and H\"older distortion of $f(z+r) - f(z)$ for any mapping defined in the entire plane; after a normalization we are reduced to \eqref{logs}. 

\begin{proof}[Proof of Theorem \ref{ptwise}]
The case of the equality for $f_\tau$ is immediate. Note, in particular,  that the circle \eqref{etas} is precisely the set of points $\tau$ such that
$$\left|\frac{\tau-1}{\tau +1}\right| =  \frac{K-1}{K+1}\, .
$$ 
Since each  $f = f_\tau$ in \eqref{eq:gammaalpha} satisfies the Beltrami equation $f_{\bar z} = \frac{\tau-1}{\tau +1} \frac{z}{\bar z} f_z$, the mappings are indeed  $K$-quasiconformal. 

Next, let $0 < r \leqslant 1$, and consider first $K$-quasiconformal mappings which are conformal outside the disk $B(0,2)$ with 
\begin{equation}
\label{taas}
 f(z) =  z + {\mathcal O}(1/z)\qquad \mbox{as } |z| \to \infty. 
\end{equation}
Embed then $f$ to a holomorphic flow 
$\{f_{\lambda} \}_{\lambda \in \DD}$ as in \eqref{eq:flow}, so that each map has  the same asymptotics  \eqref{taas}.  By Koebe distortion (see e.g. \cite[Theorem 2.10.4]{AIMb}) $f_\lambda B(0,2) \subset B(0,4) $, so that the values of $f_\lambda(z) - f_\lambda(0)$ for $|z| \leqslant 2$ all stay within the disk $B(0,8)$. We thus observe that the holomorphic function 
$$\psi(\lambda) := \log \frac{8}{f_\lambda(r) - f_\lambda(0)}, \qquad \lambda \in \DD,
$$
where the  branch for the logarithm is fixed by setting $\psi (0)>0$, has positive real part. 

These preparations take us to the main point of the argument, where we simply apply  Schwarz lemma to $\psi: \DD \to \Hpl$, with $\Hpl$ the right half plane. This implies  $\left|\psi(\lambda) - \frac{1+|\lambda|^2}{1-|\lambda|^2}\, \psi(0) \right| \leqslant \frac{2|\lambda|}{1-|\lambda|^2}\, \psi(0) $. Unwinding the definitions  gives us 
\begin{eqnarray}
\label{prince}
\hspace{1cm}  \left| \log [f(r) - f(0)]  - \frac12 \left(K+\frac{1}{K} \right) \,  \log r \right| &\\ \nonumber \\
  \leqslant  \,\frac12 \left(K-\frac{1}{K} \right)  \log \frac{1}{r} & + \;\;(K-1)\log 8  \nonumber 
\end{eqnarray}
\vspace{-.01cm}
Thus a version of  \eqref{logs} is shown to hold for the principal branch.

It remains to find estimates when we do not have  conformality outside $2\DD$. For this, given a quasiconformal mapping $f$ of $\C$ normalized by $f(0)=0$ and  $f(1)=1$, apply Stoilow factorization (\cite[Theorem 5.5.1]{AIMb}) and represent $f = F \circ f_0$, where $f_0$ is a principal map conformal outside the disk $2\DD$, and $F$ a quasiconformal map which is conformal in the domain $\Omega = f_0(2\DD)$.  

Here we will make use of the following lemma.

\begin{lemma} 
\label{riemann}
Suppose  that $F \colon \Omega \to \C$ is conformal, where $\Omega \subsetneq \C$ is simply connected. Then for any $x, y, w \in \Omega$  with $x\not=w$, $y\not= w$, and any chosen branch of the logarithm {\rm (}c.f. Proposition \ref{branch5}{\rm )} it holds that
$$ \left| \log \frac{F(x)  -  F(w)}{x-w} \, - \,  \log \frac{F(y) - F(w)}{y-w} \right| \leqslant 10 \rho_\Omega(x,y).
$$
where $\rho_\Omega$ is the hyperbolic metric of the domain $\Omega$.
\end{lemma}
\begin{proof}  The lemma is basically a reformulation of the classical distortion bounds for conformal mappings in the unit disk. Indeed, suppose first that $F$ is conformal in $\DD$ with $F(0)=0$.  Then  \cite[p.21]{Pomm} we have $\left| \frac{F''(z)}{F'(z)} \right| \leqslant 6(1- |z|^2)^{-1}$  uniformly in $\DD$ which implies  
for any branch of the logarithm that
$$|\log F'(x) - \log F'(y)| \leqslant 3 \rho_\DD(x,y), \qquad x, y \in \DD.$$ On the other hand \cite[p.66]{Pomm}, we also have
$$ \left| \log \frac{z F'(z)}{F(z)} \right| \leqslant \log \frac{1+|z|}{1-|z|}, \qquad z \in \DD,
$$  
with the normalization $\log 1=0$ at $z=0.$
Combining the estimates gives, again for any branch of the logarithm, 
\begin{equation} \label{class3} \left|  \log \frac{F(x)}{x} -  \log \frac{F(y)}{y} \right| \leqslant 5\rho_\DD(x,y), \qquad x, y \in \DD. 
\end{equation}

In the general case compose with the Riemann map $\phi: \DD \to \Omega$, $\phi(0) = w$. 
Then for any choices of the respective branches
\vspace{-.3cm}

\begin{eqnarray}
\label{logcomp}
 \hspace{-.94cm} \log \frac{F(\phi(z))-F(\phi(0))}{\phi(z)-\phi(0)}  &=&  \log \frac{F(\phi(z))-F(\phi(0))}{z} -  \log \frac{\phi(z)-\phi(0)}{z} \nonumber \\
 &+& \;2\pi  i \, n, \qquad 0 \neq z \in \DD, \nonumber
\end{eqnarray}
 for some $n\in \Z$. Applying now \eqref{class3} twice proves the claim.
\end{proof}

Returning to the proof of Theorem \ref{ptwise} and the factorization $f = F \circ f_0$,  the  triple $x= f_0(r), y=f_0(1)$, $w=f_0(0)$ is contained in $ f_0( \overline{\DD} )$, and this set  has hyperbolic diameter  in $\Omega =  f_0(2\DD)$ bounded uniformly in terms of $K$ only.
As $F(x) = f(r), F(y) = 1$ and $F(w) = 0$, Lemma \ref{riemann} shows that
$$ \Bigl| \log f(r) - \log\left( f_0(r) - f_0(0) \right) + \log\left(f_0(1) - f_0(0) \right)  \Bigr| \leqslant  M(K) < \infty.
$$
On the other hand, applying \eqref{prince} with $r = 1$ to $f_0$ implies 
\begin{equation}
\label{ysi}
|\log\left(f_0(1) - f_0(0)\right) | \leqslant (K-1)\log 8. 
\end{equation}
Finally, combining these bounds with \eqref{prince}  for a general $0 < r < 1$  completes the proof of the theorem.
\end{proof}

\subsection{The Infinitesimal Rate of Rotation} \label{infit}

For a quasiconformal mapping $f:\Omega \to \Omega'$ defined in a proper subdomain $\Omega \subset \C$ it is senseless to ask for uniform bounds for  rotation or stretching in all of $\Omega$ - the  general properties  are determined by the geometry of the specific domains in question.  One should recall in this context that even  a conformal map from unit disc onto e.g. a snowflake domain rotates in an unbounded manner  as one approaches suitably the boundary. 

On the other hand, in any domain  for {\it  infinitesimal} rotation and stretching we have uniform bounds.
To describe this more precisely, first consider for $\tau: =\alpha(1+i \gamma)$  the prototype example 

\begin{equation} \label{model12}
f_\tau(z) =  \frac{z}{|z|} |z|^{\alpha(1+i\gamma)}, \;\; \mbox{ for }\;  |z| \leqslant 1 \; \;\;\mbox{ with }\;\; \; f_\tau(z) = z, \; \mbox{ for }\; |z| >1.
\end{equation}
\smallskip

\noindent  The map $f_\tau$ takes a radial segment in the unit disk to  a logarithmic spiral with  rate of rotation  $\gamma \in \R$, see Figure  \ref{fig:spiral5}, while $\alpha > 0$\,  can be considered as the  
 stretching exponent of $f_\tau$ at the origin.
 
 \begin{figure}
\includegraphics[width=6cm]{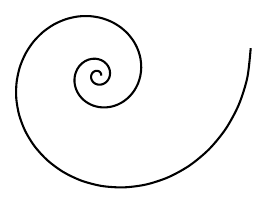}
\caption{A logarithmic spiral with rotation rate $\gamma=5$.}
\label{fig:spiral5}
\end{figure}
 
Intuitively, keeping the model map $f_\tau$ in mind, for a general mapping $f$ the local exponent  for stretching at a point $z_0 \in \Omega$ should then be given by 
\begin{equation}
\label{modela}
 \alpha(z_0) = \; \lim_{r \to 0} 
\, \frac{\; \log \big|f(z_0 + r) - f(z_0) \big|\; }{\log r}\,,
\end{equation}
while the (geometric)  rate of  rotation one would define as 
\begin{equation}
\label{modelg}
 \gamma(z_0) 
 = \lim_{r \to 0} \,
\frac{\; \arg\bigl(f(z_0 + r) - f(z_0)\bigr) \; }{\;\vspace{.1cm} \log | f(z_0 +r) - f(z_0)| \;}.
\end{equation}
In  defining the latter expression one takes into account  that  the more the mapping is compressing the more it has to rotate to produce  a given geometric spiral as the image of  a  half-line emanating from $z_0$. An easy way to illustrate this is to consider  the curve $t\to t^{1+i\gamma}$ and the change of variable $t = s^K$.

\medskip

The problem with the above  intuitive notions is, of course, that the respective limits need not exist as $f$ may have different behaviour at different scales. Therefore we formulate the  infinitesimal joint stretching-rotation bounds as follows.

\begin{theorem} 
\label{localrotstret}
Assume $f:\Omega \to \Omega'$  is a $K$-quasiconformal homeomorphism between two planar domains,  and that $z_0 \in \Omega$ is given.

Suppose we can find radii $r_k \to 0$, such that the limits 
\vspace{-.3cm}

\begin{eqnarray}
\alpha &=& \lim_{k\to \infty} \frac{\; \log \big|f(z_0 + r_k) - f(z_0) \big|\; }{\log r_k}\,, \label{alffa}\\
\gamma &=&  \lim_{k\to \infty}  \frac{\; \arg\bigl(f(z_0 + r_k) - f(z_0)\bigr) \; }{\;\vspace{.1cm} \log | f(z_0 + r_k) - f(z_0)| \;} \label{gamma}
\end{eqnarray}
\smallskip
\vspace{-.4cm}

\noindent both exist. Then  the exponent $\tau= \alpha(1+i\gamma)$ satisfies
\vspace{-.1cm}

\begin{equation}
\label{etas2}
\left| \, \tau -  \,\frac12 \left(K+\frac{1}{K} \right) \right| \leqslant \frac12 \left(K-\frac{1}{K} \right).
 \end{equation}
 \end{theorem}
 
\begin{figure}
\includegraphics[width=6.5cm]{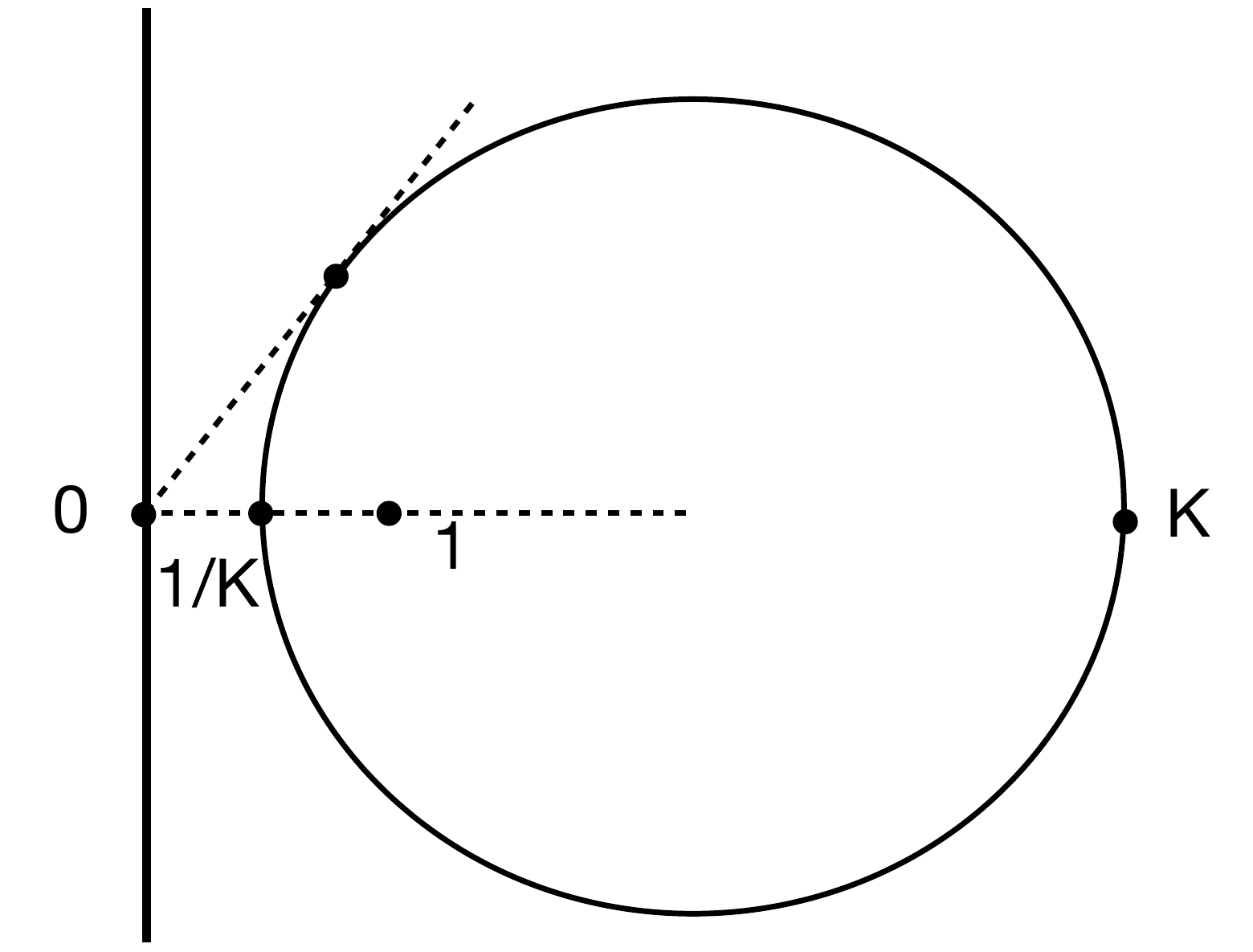}
\caption{The exponent $\tau$ from \eqref{etas2} lies in the hyperbolic disk pictured above.}
\label{fig:hypdisk}
\end{figure}

\medskip
\begin{remark}\label{rem:explicate}
To explicate the statement above, in \eqref{gamma} we use  some continuous branch of  argument $r \mapsto \arg\bigl(f(z_0 + r) - f(z_0)\bigr)$. However, it is  important to notice that the limit  \eqref{gamma}, if  exists, is independent of the choice of branch used.
\end{remark}

\begin{corollary}\label{rem:rateofrot} In particular, for every $K$-quasiconformal mapping \break $f:\Omega \to \Omega'$ and for every point  $z_0 \in \Omega$,  each limiting rate of rotation  $\gamma$ in  \eqref{gamma} satisfies
\begin{equation}
|\gamma| \leqslant \frac12 \left(K-\frac{1}{K} \right).
\end{equation}
\end{corollary}

\noindent The tangent line in Figure \ref{fig:hypdisk} illustrates the extremal case in Corollary \ref{rem:rateofrot}.

Recall furthermore that  for $\left|\frac{\tau-1}{\tau +1}\right| =  \frac{K-1}{K+1}\,$  the  model map
\eqref{model12} is  $K$-quasiconformal,  that for $z_0 = 0$ the exponents  $\alpha$ and $\gamma$ are given by \eqref{alffa} and \eqref{gamma}, respectively,  for any sequence of radii $r_k \to 0$,   and that the exponent $\tau= \alpha(1+i\gamma)$ satisfies
  \eqref{etas2} as an equality. Thus the bounds of  Theorem \ref{localrotstret} are optimal.
\smallskip

\noindent {\it Proof of Theorem \ref{localrotstret}.} It suffices to consider simply connected domains 
$\Omega \subset \C$. We first apply  Stoilow factorization 
\cite[Theorem 5.5.1]{AIMb} and represent
$$ f(z) = \varphi \circ f_0(z), \qquad z \in \Omega,
$$
where $f_0 :\C \to \C$ is $K-$quasiconformal and Ê$\varphi$ is conformal in the domain $\Omega'' = f_0(\Omega)$. If $\Omega = \C$ we take $\varphi(z) = z$.

Theorem \ref{ptwise} now implies

\begin{eqnarray}\label{dupl}
\left|\,  \log\left[\frac{f_0(z_0 + t r_0) - f_0(z_0)}{f_0(z_0 + r_0) - f_0(z_0)} \right]  -   \,\frac12 \left(K+\frac{1}{K} \right)\log t \,, \right|  &  \\
 \leqslant  \frac12 \left(K-\frac{1}{K} \right)\, \log \frac1t \;  + &  \hspace{-.1cm} c(K),  \qquad t \in (0,1],
 \nonumber
\end{eqnarray}
where $r_0$ is such that $B(z_0,r_0) \subset \Omega$.
On the other hand,  Lemma \ref{riemann} gives
\begin{equation} \label{eq:cor2.3}
\left| \log \frac{\varphi(x)  -  \varphi(z)}{x-z} \, - \,  \log \frac{\varphi(y) - \varphi(z)}{y-z} 
\right| \leqslant 10 \rho_{\Omega''}(x,y).\nonumber
\end{equation}
for $x,y,z \in \Omega''$.
Choosing
 $x= f_0(z_0 + t r_0)$, $y = f_0(z_0 + r_0)$ and $z =  f_0(z_0)$ shows that \eqref{dupl} holds for $f$, too, with $c(K)$ replaced by a constant depending on $K$ and the hyperbolic distance in $\Omega''$ between  $ f_0(z_0 + t r_0)$ and $ f_0(z_0 +  r_0)$. Adjusting $r_0$ this distance can be made arbitrarily small uniformly in $t\in (0,1]$. 
 
 Finally, letting $t = r_k/r_0$, dividing by $\log (1/r_k)$ and taking $k \to \infty$ gives the bound \eqref{etas2}. Note here that by Mori's theorem we necessarily have $\alpha \geqslant 1/K > 0$.
 \hfill $\Box$
 \smallskip

\subsection{Global Aspects of Rotation: The Complex Logarithm $\log f_z$}\label{se:complex log}

With Theorems \ref{ptwise} and  \ref{localrotstret} one has  a first description of  the  rotation and stretching properties of quasiconformal mappings.  Theorem  \ref{localrotstret} determines the extremal infinitesimal  phenomena, and it is of interest only at the points of non-differentiability.  

For a principal mapping  with $f(z) = z+o(1)$   when $z\to\infty $,  it is also natural  to consider a global notion, the rotation the image of a half line from $\infty$ to $z_0$ makes around the image point $f(z_0)$, see Figure \ref{fig:winding}. This global measure of rotation, however,  works best  at points of differentiability  and as it turns out,  is described by  the function  $\log  f_z$.

 Before embarking into the several interesting analytic and geometric properties that 
 the complex logarithm of $f_z$ encodes,  we of course need to clarify the pointwise definition of  the function $\,\log f_z\,$. There are actually  two alternative ways for this: an analytic approach and a geometric one. 
\vspace{.3cm}

\noindent{\it Analytic definition}.  
 As explained in Subsection \ref{bilipqc} above, by letting $\mu(z) = \mu_\lambda(z)$ depend holomorphically on a  parameter $\lambda \in \DD$, we have  natural ways to embed a given principal mapping $f$  into a family of quasiconformal maps $f^\lambda$, with  analytic dependence on the parameter $\lambda \in \DD.$  To make further use of this property, we need
 a   refinement of  \cite[Lemma 3.7]{AIPS}.
\begin{lemma}\label{le:simple}
There is a set $E \subset \C$ of full measure, such that for every $\lambda \in \DD$ and for every $z \in E$, the map $ f^\lambda$ is differentiable at $z$ with  $f^\lambda_z \neq 0$. 

In addition, for a fixed $z\in E$  and for any compact $A\subset \DD$ the differential quotients $$ \frac{f^\lambda(z+t)-f^\lambda(z)}{t}, \qquad t\in (0,1), \; \lambda\in A,$$ are uniformly bounded. 

Furthermore, for $z \in E$ the function $\lambda \mapsto  f^\lambda_z(z)$ is  holomorphic in $\lambda \in \DD$ and the equation 
 {\rm \eqref{eq:beltrami2}} with $\mu = \mu_\lambda$ holds true pointwise in $E$.  
\end{lemma}
\noindent We postpone the  proof  of the Lemma to the end of this section.
As $f^0(z) \equiv z $,  it is natural to set $\log f_z^0 \equiv 0$, and thus requiring that  $\log f^\lambda_z(z)$ remains holomorphic in $\lambda$ defines the logarithm uniquely at every $z \in E$.
Returning to the original solution $f$ we can now set the analytic definition
\[ \log f_z(z) = \log f^k_z(z), \qquad k = \| \mu \|_{\infty}, \quad z \in E.
\]

\vspace{.2cm}
\noindent{\it Geometric definition}.  While the analytic approach to $\log f_z$ requires mappings defined in the entire plane, there is a  more geometric approach that works in  general settings. Indeed, by Proposition \ref{branch5} given a quasiconformal mapping $f:\Omega \to \Omega'$ between simply connected domains, {\it each} branch of the logarithm
 \begin{equation}
\label{logs44}
\qquad \; \log \frac{f(z)-f(w)}{z-w}\,, \qquad 
(z,w) \in  \Omega \times \Omega, \quad  z \neq w,
\end{equation}
determines at  points  
of differentiability of $f$ (such that also  $f_{\bar z} (z)= \mu (z) f_z(z)$, hence almost everywhere) 
a corresponding branch of the complex logarithm of $f_z$, namely
 \begin{equation}
\label{logs7}
 \log \bigl( \partial f(z)  \bigr) = \lim_{t\to 0^+}\;   \log \left[\frac{f(z + t) - f(z)}{t} \right] - \log[ 1 + \mu(z)].
\end{equation}
The role of the auxiliary horizontal approach  is of course redundant, and one could as well apply an approach  from any other direction,
\vspace{-.5cm}

$$ \log \bigl( \partial f(z)  \bigr) = \lim_{t\to 0^+}  \log \left[\frac{f(z + te^{i\theta}) - f(z)}{t} \right] - \log[ 1 + e^{-2i\theta} \mu(z)] - i\theta
$$
\vspace{.01cm}

For general  $f:\Omega \to \Omega'$ there is no natural choice for the argument of $\partial f(z)$ in \eqref{logs7} but for principal mappings $f:\C \to \C$, unless otherwise explicitly stated, we will always consider 
 the principal  branch of $\log f_z(z_0)$, determined by the conditions \eqref{branch} and \eqref{logs7}. This branch  admits a clear geometric interpretation.
Namely, 
we may  consider  the corresponding  continuous branch of the argument,
\begin{equation}
\label{arg}
\arg\bigl[f(z_0 + t) - f(z_0) \bigr] = \im  \log \left[\frac{f(z_0 + t) - f(z_0)}{t} \right] , \qquad 0 < t < \infty,\nonumber 
\end{equation}
where in the case of the principal branch \eqref{branch} we have  $ \arg\bigl[f(z_0 + t) - f(z_0) \bigr] \to 0$ as $t \to +\infty$. 
The integer part of $\frac{1}{2\pi} \arg\bigl[f(z_0 + t_0) - f(z_0) \bigr]$ then gives the number of times the point $f(z) = f(z_0 + t)$ winds around  $f(z_0)$, when $z= z_0 + t$ moves from $t= +\infty$ to $t=t_0$ along the horizontal line through $z_0$.  
Further, at any point  $z_0$ of differentiability of $f$   the  argument 
satisfies
\begin{eqnarray*}
\arg\bigl[f(z_0 + t) - f(z_0) \bigr]  & = &  \arg\left[\frac{f(z_0 + t) - f(z_0)}{t} \right]  \\
 & = &  \arg\bigl[\partial f(z_0) + \mu(z_0) {\partial} f(z_0) + {\mathcal O}(t) \bigr]
\end{eqnarray*}
as $t \to 0^+$. Since $|\mu| \leqslant k < 1$, we may choose $\arg[ 1 + \mu(z)] \in (-\frac{\pi}{2}, \frac{\pi}{2})$, and obtain almost everywhere the {\it geometric definition} of the argument \;
$ \arg\left[\partial f(z_0) \right] = \im   \log \partial f(z_0)$.

As the above discussion indicates, with this definition {\it at points of differentiability, for the principal branch we have   $|\im \, \log(\partial f)(z) - 2 \pi\, n| < \pi$ where $n$ is the number of times (with sign) that $f(w)$ winds around $f(z)$ when $w$ travels from $\infty$ to $z$ along a fixed radius, see Figure \ref{fig:winding}.} 
\medskip

\begin{figure}
\includegraphics[width=12cm]{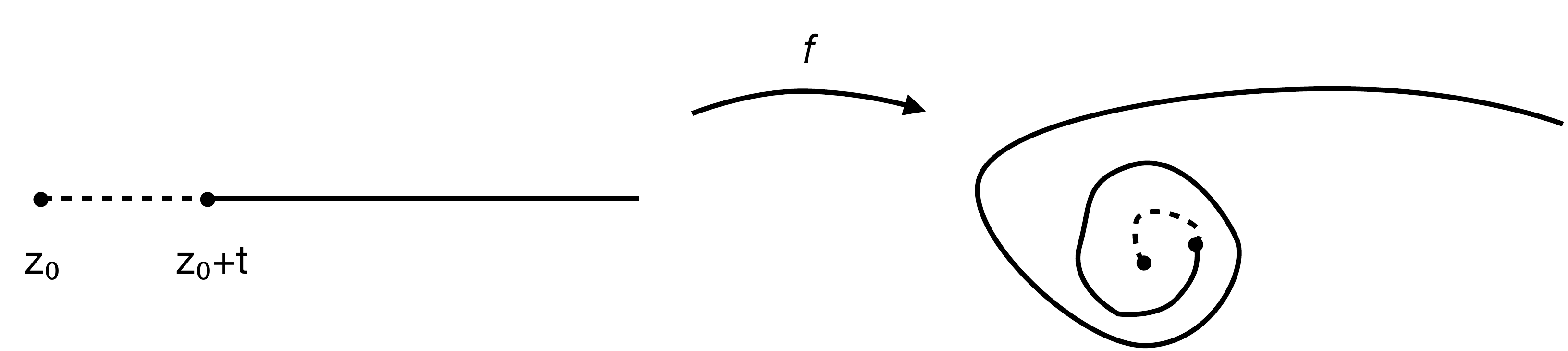}
\caption{$\arg \partial f(z_0)$ measures the total winding of the image curve around the point $f(z_0)$.}
\label{fig:winding}
\end{figure}

Even if above we used two quite different approaches to the complex logarithm of $f_z$, both methods lead to the same concept.

\begin{lemma}\label{le:agree}
 For any principal quasiconformal mapping $f:\C \to \C$, the geometric and analytic definition of  the principal branch of $\log f_z$ agree on a set of full measure.
\end{lemma}
\begin{proof} 
\smallskip

It is useful to observe that  the definition of the principal branch of $\arg\bigl[f(z_0 + t) - f(z_0) \bigr] $ given above 
can  equivalently be  obtained from the imaginary part of the analytic continuation 
$h(\lambda)=\log (f^\lambda (z_0 + t) - f^\lambda (z_0))$, when one sets
$h(0)=\log t.$

For the proof we only need to consider the  set $E \subset \C$ given by Lemma \ref{le:simple}. At any $z\in E$ define
$$ a_t(\lambda, z) =  \log \left[\frac{f^\lambda(z + t) - f^\lambda(z)}{t} \right] - \log[ 1 +  \mu_\lambda(z)]
$$
Then for a fixed $z \in E$, $\lim_{t\to 0+} a_t(\lambda,z)$ is holomorphic in $\lambda$, as a pointwise limit of bounded holomorphic functions (here we use the statement of uniform boundedness of the differential quotients in Lemma \ref{le:simple}). Further, the function vanishes at $\lambda = 0$ and one has $\exp\bigl(\lim_{t\to 0+} a_t(\lambda,z)\bigr)= \partial f^\lambda(z)$ for  every $\lambda \in \DD$.
\end{proof}

In particular, the definition of $\log f_z$  does not depend on how the mapping is embedded to a holomorphic flow.
\smallskip

Once the pointwise notion of   $\log f_z$ is settled,  one can ask for its function space properties. Here it follows from the work of  D. Hamilton \cite{Ham} that  $\log f_z \in  BMO(\C)$. He complexifies  an argument of  Reimann \cite{Reim} and estimates the $\lambda$-derivative 
 of the function. Indeed, writing $f^{\lambda + \epsilon} =  h_\epsilon \circ f^\lambda$ and using \cite[Thm. 5.5.6]{AIMb}  we have 
 \vspace{-.3cm}
 
 \begin{equation}
\mu_{h_\epsilon} \left(f^\lambda(z)\right) =\;  \epsilon \,  \frac{\mu \, \| \mu \|_\infty}{\| \mu \|_\infty^2 - |\lambda|^2 |\mu|^2}\,  \frac{f^\lambda_z}{\; {\overline{ f^\lambda_z}}\; } \; + \;  {\mathcal O}(\epsilon^2) =: \epsilon \, \nu_\lambda \circ f^\lambda(z) \; + \;  {\mathcal O}(\epsilon^2)
\end{equation}
For the derivatives of $h_\epsilon$ we have at almost every point $( h_\epsilon)_{ z} = 1 + \epsilon S(\nu_\lambda)  \; + \;  {\mathcal O}(\epsilon^2)$  with 
$ ( h_\epsilon)_{\bar z} = \epsilon \nu_\lambda  \; + \;  {\mathcal O}(\epsilon^2)$, where $S$ is the Beurling transform, see \cite[Section 5.7]{AIMb}. Thus, from the chain rule we have
\vspace{-.2cm}

\begin{eqnarray}
 \partial_z f^{\lambda + \epsilon} -\partial_z  f^{\lambda}  & = &  \partial_z  h_\epsilon (f^\lambda) \, \partial_z f^\lambda - \partial_z f^\lambda +  \partial_{\bar z}  h_\epsilon (f^\lambda) \, {\overline{\partial_{\bar z} f^\lambda}}  \nonumber \\ 
 & = & \epsilon \, (S\nu_\lambda)(f^\lambda) \, \partial_z f^\lambda +  \epsilon \, \nu_\lambda(f^\lambda)  \, {\overline{\partial_{\bar z} f^\lambda}}\; + \;  {\mathcal O}(\epsilon^2) \nonumber
\end{eqnarray}
Thus
\vspace{-.2cm}

$$\frac{\partial_\lambda f_z^\lambda}{f_z^\lambda} =  (S\nu_\lambda)(f^\lambda) + \nu_\lambda(f^\lambda)  \, \overline{  \, \mu_\lambda\, } \,\frac{\; {\overline{ f^\lambda_z}}\;  }{f^\lambda_z}\,  = \,  (S\nu_\lambda)(f^\lambda)\, + \, \, \overline\lambda \, \,\frac{ |\mu|^2}{\| \mu \|_\infty^2 - |\lambda|^2 |\mu|^2}
$$
\smallskip
\vspace{-.2cm}

\noindent Since  by another theorem of Reimann,  quasiconformal mappings preserve the space  $BMO(\C)$, the right hand side has $BMO$-norm uniformly bounded for $|\lambda| \leqslant k_0 < 1$. By integrating along a radius from $\lambda = 0$ to $\lambda = \|\mu\|_\infty$ we obtain

\begin{proposition}{\rm \cite{Ham}} \label{prop:hamilton}
Suppose $f\in W^{1,2}_{loc}(\C)$ is a  principal  quasiconformal mapping. Then
$$ \log f_z \in BMO(\C).
$$
\end{proposition}
\smallskip

In particular, $\arg f_z$ is exponentially integrable, i.e. $$e^{b|\arg f_z|} \in L^1_{loc} \qquad \mbox{for some positive constant} \;  b.$$ One of the key points of the present  work is  finding the optimal form of the exponential integrability. This will be done in Corollary \ref{cor:argumentintegrability} where we give  the sharp bounds for  $b$ in terms of the distortion $K(f)$.
\medskip

To complete the subsection we need to  justify Lemma \ref{le:simple}.

\begin{proof}[Proof of Lemma \ref{le:simple}]
After  \cite[Lemma 3.7]{AIPS} one only needs to verify the (simultaneous and uniform for all $\lambda$)  differentiability almost everywhere. Recall that in plane continuous functions $f\in W^{1,p}_{loc}(\C)$ with $p>2$ are differentiable almost everywhere. A proof is based on  the classical Morrey estimate (see \cite[Section 4.5.3]{EG}),
\begin{equation}\label{morrey}
|g(y)-g(x)|\leqslant c_p|y-x|\left(\frac{1}{B(x,|y-x|)}\int_{B(x,|y-x|)}|Dg|^p\right)^{1/p}, \quad g\in W^{1,p}_{loc}(\C), \nonumber 
\end{equation}
where $p>2.$
The differentiability a.e.  follows by applying the above estimate to the function $
g(y)= f(y)-f(x)-Df(x)(y-x)$ and using a Lebesgue point argument. 

For our purposes we only need to quantify this a little bit. Denote
by $Q_f(x)$ the maximal difference quotient at $x$, given by
 $$Q_f(x):= \sup_{|y-x|\leqslant 1}\frac{|f(y)-f(x)|}{|y-x|}.$$ As a direct consequence of Morrey's estimate, any continuous $f\in W^{1,p}_{loc}(\C)$ satisfies the pointwise estimate
$$
Q_f(x)\leqslant c_pM_p(Df)(x),
$$
where  $M_p(g)(x):=\sup_{r>0} \left(\frac{1}{B(x,r)}\int_{B(x,r)}|g|^p\right)^{1/p}.$ Especially, when 
derivatives of $f$ 
belong to $ L^{p}(\C)$ 
the weak (1,1)-continuity of the maximal function yields for any $s >0$
\begin{equation}\label{morrey2}
|\{ Q_f> s\}|\leqslant c'_p \; s^{-p}\|Df\|_p^p.
\end{equation}

In our situation,   for any given $\delta <1$  we may pick $p = p(\delta) >2$ so close to 2 that the standard $\,\mathscr L^p$ --properties of the Beurling operator and the Neumann series representation (see \cite[p. 163]{AIMb}) enable one to write, for  any $\lambda \in \DD$,
\begin{equation}\label{morrey3}
f^{\lambda}(z)-z=\sum_{n=1}^\infty\lambda^n f_n(z)\quad {\rm with}\;\; \|Df_n\|_p\leqslant c_0 \, \delta^{-n/2}.
\end{equation}
Here  $c_0$ depends only on the size of  the support of $\mu$ in \eqref{eq:flow},  and  each $f_n\in C(\C) \cap W_{loc}^{1,q}(\C)$ for every $q<\infty$. The series converges locally uniformly in $\C$.


For $n\geqslant1$ write $h_{n,\lambda}:=\sum_{k=n}^\infty\lambda^k f_k$ . Fix $\delta <1$, denote $A_k:=\{ Q_{f_k}> \delta^{-3k/4}\}$, and use (\ref{morrey}) to estimate  $|A_k|\leqslant c_0^p \, \delta^{-kp/2}\delta^{3kp/4}= c \,\delta^{kp/4}.$ Next, write  $F_n:=\bigcup_{k=n}^\infty A_k$, so that
 $|F_n|\lesssim \delta^{np/4}$ while if $x\in F_n^c$, we have
\begin{equation}\label{morrey2''}
Q_{h_{n,\lambda}}(x)\leqslant \sum_{k=n}^\infty \delta^k \delta^{-3k/4} \lesssim \delta^{n/4}
 \quad \mbox{for every }\; |\lambda| \leqslant \delta. 
\end{equation}
Clearly this  implies that  the set
$E_{\delta}:= \cap_{n=1}^\infty F_n$ has full measure with $ \lim_{n\to \infty} Q_{h_{n,\lambda}}(x) =0$ for every $|\lambda| \leqslant \delta$ and $x \in E_\delta$.  

 Pick now a set $\widetilde E$ of full measure such that each $f_k$ in (\ref{morrey3}) is differentiable at every point of $\widetilde E$. 
 Especially, for every $|\lambda| \leqslant \delta$ \, the function $f^{\lambda}$  is differentiable at each $x\in \widetilde E\cap E_\delta,$ with
 $$ Df^{\lambda}(x) = Id +  \sum_{k=1}^\infty \lambda^k Df_k(x).
 $$
 Whence we may choose $E:=\widetilde E\cap\bigcup_{\ell=1}^\infty E_{1-1/\ell}.$ Finally,
the statement about uniform boundedness of the differential quotients with respect to $\lambda$ in compacts follows directly from \eqref{morrey2} and \eqref{morrey2''}.
\end{proof}

\section{Interpolation with complex exponents and Burkholder integrals}\label{se:burkholder}

\subsection{Interpolation with complex exponents} 

Let $\,(\Omega,\sigma)\,$ be a measure space and let $\mathscr M(\Omega,\sigma) \,$ denote the space of  complex-valued $\,\sigma$ -measurable functions on $\Omega$. We consider  $\,\mathscr L^p (\Omega,\sigma)\,$ spaces in which the (quasi-)norms are defined by
$$
  \| \Phi\|_p  = \left(\int_\Omega |\Phi(x)|^p \;\textrm{d}\sigma(x)\right)^{\frac{1}{p}}\,,\;\;\;0< p<\infty\;, \;\;\; \textrm{and}\;\;\|\Phi\|_\infty \;= \underset{x\in\Omega}{\textrm{ess\,sup}} \,\;|\Phi(x)|.
$$
We shall consider analytic families $\Phi_\lambda$ of measurable
functions in $\Omega$, i.e. jointly measurable functions $ (x, \lambda) \mapsto \Phi_\lambda (x)$ defined on $\Omega \times U$, where $U\subset\C$ is a domain, and
such that for each fixed $x \in\Omega$  the map $\lambda \mapsto \Phi_\lambda(x)$ is analytic in $U$. The family is
said to be \emph{non-vanishing} if there exists a set $E \subset\Omega$ of $\sigma$-measure zero such that $\Phi_\lambda(x)\not=0$ 	for all $x \in\Omega\setminus E$ and for  all $\lambda \in U$. 

In the following we give a generalization  of the
interpolation results from \cite{AIPS} to  complex exponents. In fact, restricting below  in \eqref{lisa66} to the real interval $\beta \in \left[\, p_0, \, \frac{p_0}{2}\left(1+ \frac{1}{|\lambda|}\right)\right]$ one arrives at the bounds of  \cite[Lemma 1.6]{AIPS}.

\begin{lemma}[Interpolation Lemma with complex exponents]\label{le:interpolation}
Suppose  \\ $\;\{\Phi_{\lambda}\,;\; |\lambda| < 1\}\, \subset\,\mathscr M
(\Omega, \sigma)\,$ is a non-vanishing analytic family of functions, parametrized by  complex numbers $\lambda \in \DD$, such that for some  $p_0>0$,
\[ \Phi_0 \equiv 1, \quad \mbox{and} \quad  \| \Phi_\lambda \|_{p_0} \leqslant 1 \quad \text{ for every } \, \lambda \in \DD.
 \]

Then, for every $|\lambda|<1$ and for every complex  exponent $\beta \in \mathbb{C}$ contained in the closed ellipse
\begin{equation}
\label{lisa66}
 \, |\beta| + \, |\beta - p_0| \leqslant \frac{p_0}{|\lambda|}, 
\end{equation}
  we have
\begin{equation}\label{eq:betaint}
\int_\Omega \left| \Phi_\lambda^{\; \beta}  \right| \; \textrm{d} \sigma \leqslant 1.
\end{equation}
The choice of branch in \eqref{eq:betaint} is the natural one, determined by the condition $\log \Phi_0 = 0$. 
\end{lemma}
\medskip

\begin{proof} 
By considering the analytic family $\Phi_\lambda^{\, p_0/2}$ we may restrict our attention to the $p_0=2$ case, which will later make calculations simpler.
Observe that our assumption $\|\Phi_0\|_2 \leqslant 1$ implies $\sigma(\Omega) \leqslant 1.$ Actually, by the maximum principle for analytic $\mathscr L^2$-valued functions, we may further assume the strict inequality
\begin{equation}\label{eq:size}
\sigma(\Omega) <1.
\end{equation}
Otherwise $\Phi_\lambda$ would be constant in $\lambda $, as we have
$\|\Phi_\lambda\|_2\leqslant 
1$ for all $\lambda\in \DD.$
We may also assume in the proof that $0<c\leqslant |\Phi_\lambda(x)|\leqslant C<\infty $
uniformly for all $(x,\lambda)\in \Omega\times U$,  as the reduction of the 
general situation to this is done exactly as in \cite[Section 2]{AIPS}. Similarly, 
we choose an arbitrary positive probability density $\wp$, uniformly bounded away from $0$ and $\infty$,
 $$\, \wp\in \mathscr M(\Omega, \sigma)\,,\;\;\;
 \quad\; \|\, \wp\,\|_1 = \int_\Omega \;\wp(x)\,\textrm{d}\sigma(x)\;= 1.\;$$
By Jensen's inequality using the convexity of $x\mapsto x\log (x)$ and \refeq{eq:size} we have $I:=\int_\Omega \wp(x)\log \wp(x)\, dx>0.$
Temporarily assuming that $\,\wp\,$ is fixed, we consider the holomorphic function $f$ in the unit disk
$$
f(\lambda)=\frac{1}{I} \int \wp(x) \log \Phi_\lambda (x) \,\textrm{d}\sigma(x)\;.
$$
Again by Jensen's inequality  we have the bound
\begin{eqnarray}
\label{lasku}
2\re f(\lambda)-1 &=&\frac{1}{I} \int_\Omega \wp \log \frac{|\Phi_\lambda|^2}{\wp}\,\textrm{d}\sigma \leqslant \, \frac{1}{I} \log \left(\int_\Omega \wp\, \frac{ |\Phi_\lambda|^2}{\wp}\,\textrm{d} \sigma\right) \nonumber  \\ &=& \frac{1}{I}\;\log \|\Phi_\lambda\|_2^2  \;\leqslant 0 \;,\quad \,\textrm{for} \;\;|\lambda| < 1. \nonumber 
\end{eqnarray}
Thus $f$ maps the unit disk into a half-plane $\re f(\lambda) \leqslant 1/2$, while $f(0)=0$ by our assumption $\Phi_0 \equiv 1$. 
At this stage we appeal to Schwarz lemma and deduce that  for any $0\leqslant  k <1$, the image $f(\{ |\lambda| \leqslant k \})$ lies in a hyperbolic disk $D_k$,  centered at $0$,  of the above half-plane. Precisely,
\[
f(\{ |\lambda| \leqslant k \}) \subset  D_k=\left\{ \frac{z}{1+z} : |z| \leqslant k \right\}.
\]

Our objective is to find all  exponents $\beta \in \C$ such that for all $|\lambda| \leqslant k$,
\begin{equation}\label{eq:goal}
 \re \left(\beta f(\lambda) \right) = \frac{1}{\int \wp(x) \log \wp(x)} \re \left(\beta \int \wp(x) \log \Phi_\lambda(x) \;\textrm{d}\sigma(x) \right) \leqslant 1.
\end{equation}
Equivalently, we aim at  the estimate $$\int_\Omega \wp(x) \log \frac{|\Phi_\lambda(x)^\beta|}{\wp(x)}\,\textrm{d}\sigma(x) \leqslant 0,
$$
and once this is achieved, by specialising the choice of $\wp$, that is,  choosing  $\wp(x):=\left| \Phi_\lambda(x)^\beta  \right|\left(\int_\Omega \big|\Phi_\lambda^\beta\big|\right)^{-1}$, 
we will obtain our assertion
\[
 \log \left( \int_\Omega \left| \Phi_\lambda^\beta  \right| \;\textrm{d}\sigma \right) \leqslant 0.
\]

Given $|\lambda| \leqslant k$, the condition \refeq{eq:goal} is now ensured by
\begin{equation}
\label{left1}
 \re \bigl( \beta w\bigr) \leqslant 1, \quad \text{for all $w \in D_k$},
\end{equation}
as the range of $f(\lambda)$ lies in $D_k$.  To state this requirement more explicitly,   note  first that as
$$ \frac{z}{1+z}  =  \frac{z}{1+z} \cdot \frac{1 + \overline z}{\;1-|z|^2} - \frac{|z|^2}{1-|z|^2} \, ,
$$
 inequality  \eqref{left1} takes the form
\begin{equation}
\label{lisays}
\re  \left( \beta \; z \;   \frac{1+ \overline z}{1+z}\;  \right) - |z|^2 \; \re  \beta \leqslant 1 - |z|^2
,\qquad |z| \leqslant k.
\end{equation}
Here only the first term depends on the argument of $z$, with $ z \mapsto \; z \;   \frac{1+ \overline z}{1+z}\;$ preserving the circle of radius $k$. Thus \eqref{left1} is equivalent to
\begin{equation}
\label{lisa4}
 k \,|\beta| - k^2 \;  \re  \, \beta \; \leqslant \; 1-k^2\nonumber
\end{equation}
By simple algebra, we have here the equality if and only if 
\begin{equation}
\label{eq:repre}
\beta =  1+\frac{1}{k}  \cos \theta + i \frac{\sqrt{1-k^2}}{k}\sin \theta, \quad \mbox{ for some } \theta \in [0,2 \pi].
\end{equation} 
Thus the extremal $\beta$ lie on  the ellipse with foci $\{0,2\}$ and eccentricity $k$,  so that
   \eqref{left1} is equivalent to $k|\beta| + k|\beta -2| \leqslant  2$. 
For a general exponent $p_0>0$, the ellipse takes the form of \eqref{lisa66}.   
\end{proof}

\begin{remark} \label{kuusi} Given $\beta$, 
the left hand side of \eqref{lisays} attains its maximum over $\{ \lambda : |\lambda | = k \}$  
at the point where 
$$\arg \beta = 2 \, \arg(1+\lambda) -  \arg\, \lambda ,$$
that is when
$$\beta =   \frac{1+ \lambda}{\lambda}\, (1+\lambda) \,s(\lambda), \qquad s(\lambda) \in \R_+
$$
Testing this requirement against \eqref{left1} shows that $s(\lambda) \leqslant 1/\bigl(1 + \re  \lambda \bigr)$.
Therefore  the equality in \eqref{left1} is attained at $\lambda, |\lambda|=k$, if and only if
\begin{equation}
\label{left2}
 \beta =\frac{(1+\lambda)^2}{\lambda(1+\re \lambda)}= \left(1+ \frac{1}{\lambda} \right) \frac{1+\lambda}{1+\re \lambda}.
\end{equation}
\end{remark}
\bigskip

\subsection{Burkholder integrals}
Let $f \colon \C \to \C$ be the principal solution of a Beltrami equation
\begin{equation}
\label{eq:beltrami}
f_{\bar z}(z)=\mu(z) \, f_z(z), \qquad |\mu(z)| \leqslant k\, \chi_{\mathbb{D}}(z), \quad 0\leqslant k<1.
\end{equation}
In studying the rotation spectrum of bilipschitz or quasiconformal mappings we are faced with the question for which  exponents $\beta \in \C$ is the complex power $(f_z)^\beta$ locally integrable? It turns out that the universal bounds are given exactly by \eqref{lisa66} with $p=2$, $|\lambda| = k$, that is, in terms of an ellipse having foci $\{0,2\}$ and eccentricity determined by the ellipticity constant $k$ of the equation. See Theorem \ref{co:localcomplexintegrability} for the precise statement.

In fact,  we are going to carry out our analysis in the weighted setting and  consider  the so called Burkholder type integrals. Here recall the  functionals introduced by Burkholder \cite{Bu1}, which applied to the derivatives of a map $f:\R^2 \to \R^2$ take the form
\begin{eqnarray}
B_p(Df) & = &  \frac{1}{2} \,\Big(\; p \; J(z,f)   + ( 2 - p)\; |Df|^2 \; \Big)\cdot |Df|^{p-2} \nonumber \\
& &= \;  \; \big(\;|f_{z}|\,-\,(p-1)\,|f_{\bar{z}}| \;\big)\cdot \big(\;|f_z|\;+\;|f_{\bar{z}}|\;\big)^{p-1},\, \qquad  p \in \R. \nonumber
\end{eqnarray}

Originally the functional was discovered by Burkholder in his studies of optimal martingale inequalities - since then optimal integral identities related to $B_p$, and in particular its conjectured quasiconcavity  (for $|p-1| \geqslant1 $) have been of wide interest. For recent advances and background see e.g. \cite{AIPS}, \cite{Ban}, \cite{Iw1}. 

Here we search for  corresponding functionals determined by a complex 
 parameter $p \in \C \setminus B(1,1)$. We define an auxiliary unimodular function $\rho=\rho(z)$, $|\rho| \equiv 1$, by requiring that the complex numbers
\begin{equation}
\label{uus6}
p \, \rho(z)   \quad  \mbox{ and }  \quad   1 + \rho(z) |\mu(z)| \quad \mbox{ have the same argument.}
\end{equation}

\begin{theorem} \label{main}
Suppose  we are given  a complex parameter $p$ with $1 \leqslant |p - 1| \leqslant \frac{1}{k} $ and an exponent $\beta \in \C$ with 
    \begin{equation}
\label{uus7}
 |\beta| + |\beta -2| \leqslant  2 |p-1|.
 \end{equation}
 Then for every  principal solution to \eqref{eq:beltrami} we have
\begin{equation}
\label{uus8}
 \frac{1}{\pi} \int_{\mathbb{D}} \Bigl( \Bigl| |f_{z}| + \rho |f_{\bar z}|\Bigr|- |p|\,|f_{\bar z}|  \Bigr)
\left| \Bigl(f_{z} + \rho |\mu| f_{z} \Bigr)^{\beta-1} \right| \leqslant 1,
\end{equation}
where the unimodular function $\rho = \rho(z)$ is determined  by \eqref{uus6}.
\smallskip
\vspace{.1cm}

The estimate holds as an equality for $f(z) \equiv z$. Furthermore, when $\beta$ is determined (uniquely) in terms of $p$ by the equations
\begin{equation}
\label{eq:betap}
|\beta|+|\beta-2|=2|p-1| \quad \mbox{ and }Ê\quad \re(\beta/p)=1,
\end{equation}
we  have the equality in \eqref{uus8} for every power map of the form
$$ \hspace{1.5cm}f(z) = 
\frac{z}{|z|} |z|^{\frac{1-\eta}{1+ \eta}}, \hspace{.8cm}  \mbox{where \, \,}   
 p \frac{\eta}{1+\eta}  \in [0,1] \quad {\rm and}\;\; 
 |\eta|=k.
$$

\end{theorem}

In other words, the functional \eqref{uus8} is quasiconcave for the parameter values
$1 \leqslant |p-1| \leqslant 1/k$, within the class of principal $k$-quasiconformal deformations.

\begin{figure}
\includegraphics[width=6cm]{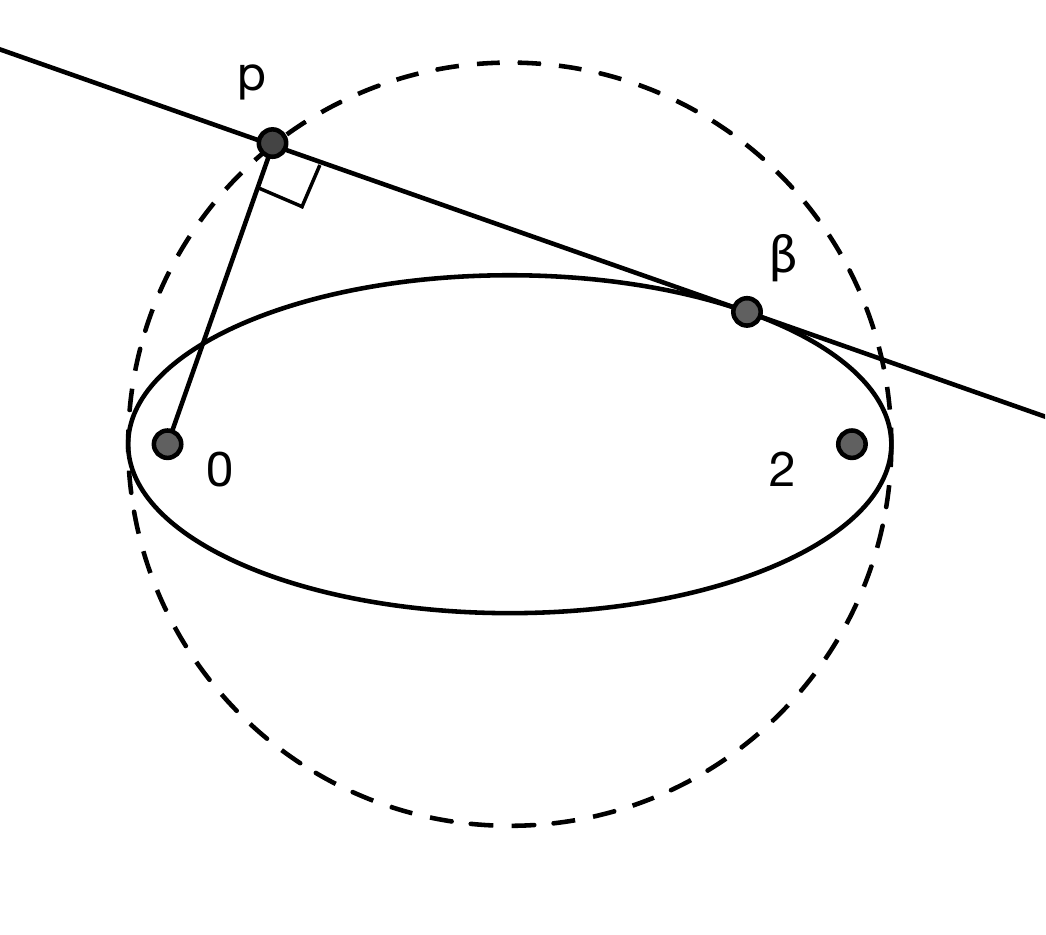}
\caption{The relation between $p$ and $\beta$ in \eqref{eq:betap}, \newline 
\centerline{$|\beta|+|\beta-2|=2|p-1|$ and $\re(\beta/p)=1$. }}
\label{fig:ellipse0}
\end{figure}

\begin{remark} Using the disk filling procedure as in \cite{AIPS}, we get many more extremals for the functional \eqref{uus8}.
\end{remark}

\begin{remark}
For $p\ge 2$  we have  $\rho \equiv 1$, while  for $p<0$, $\rho \equiv -1$. Thus for real $p$ with the choice of \eqref{eq:betap} we get back the Burkholder functionals $B_p(Df)$. On the ``phase transition'' boundary 
$|p-1|=1$, \eqref{eq:betap} forces $\beta=2$ and we recover the Jacobian.
\end{remark}

\begin{proof}[Proof of the Theorem \ref{main}] 
The proof adopts ideas from \cite{AIPS} to the case of complex exponents. 
Given the Beltrami equation \eqref{eq:beltrami},
we define a holomorphic variation as follows. For $\lambda \in \mathbb{D}$, set 
\[ \mu_\lambda(z)=\alpha_\lambda(z) \, \overline{ \rho(z) } \cdot  \frac{\mu(z)}{|\mu(z)|}, \quad \mbox{where} \quad
\frac{\alpha_\lambda(z)}{1+\alpha_\lambda(z)} = p \cdot \frac{\rho(z) |\mu(z)|}{1+\rho(z) |\mu(z)|} \cdot \frac{\lambda}{1+\lambda}
\]
Above we use the convention ``$\, 0/0=0$'' whenever dividing by zero.

By the choice of $\rho$ in  \eqref{uus6}  we have  $\; p \cdot \frac{\rho |\mu|}{1+\rho |\mu|} \geqslant 0$. Together with the assumption $1 < |p-1| \leqslant \frac{1}{k}$, in fact $ \; p \cdot \frac{\rho |\mu|}{1+\rho |\mu|} \in [0,1]$. This can be seen, for instance, by considering the half-plane $U=\{w \colon \re w <1/2 \}$ and
observing the following inequality in terms of the hyperbolic metric in $U$,
\[ d_U(0,\frac{\rho |\mu|}{1+\rho |\mu|}) \leqslant \log \frac{1+k}{1-k} \leqslant d_U(0,\frac{1}{p}).
\]
Consequently, $|\mu_\lambda(z)|=|\alpha_\lambda(z)| \leqslant |\lambda|<1$, which makes legitimate to solve the Beltrami equations
\[ f_{\bar z}^\lambda =\mu_\lambda \, f_z^\lambda, \qquad \lambda \in \mathbb{D}
\]
 under the normalization of the principal solution.

We recover the original equation for the complex value $\lambda=\frac{1}{p-1}$ and the Cauchy-Riemann equations for $\lambda=0$.

Next,   interpolate the analytic family of functions given by
\[ \Phi_\lambda(z)= \bigl(1 + \alpha_\lambda(z)\bigr) \, f_z^\lambda(z) \neq 0.
\]
Indeed, $\Phi_0 \equiv 1$ and according to Lemma \ref{le:simple}, $\{ \Phi_\lambda\}$ is a non-vanishing family in the sense required by Lemma \ref{le:interpolation}.
Furthermore,  we have the comparison
\[ \frac{J(z,f^\lambda)}{|\Phi_\lambda(z)|^2} = 1-2 \, \re \, \frac{\alpha_\lambda(z)}{1+\alpha_\lambda(z)} \ge 1-p \cdot \frac{\rho(z) |\mu(z)|}{1+\rho(z) |\mu(z)|}=: \omega(z).
\]
Thus by the classical area theorem, see e.g.\cite[p. 41]{AIMb}, we have the $\mathscr L^2$-bounds
$$  \frac{1}{\pi} \int_{\mathbb{D}} |\Phi_\lambda|^2 \; \omega \leqslant 1,
$$
and we may apply Lemma \ref{le:interpolation} with the measure space $\mathscr M(\DD, \frac{1}{\pi}w\, \mathrm{d}z)$  to obtain
\begin{equation} \label{perusmuoto}
 \frac{1}{\pi} \int_{\mathbb{D}} \Bigl(1 - p \frac{\rho |\mu|}{1+\rho |\mu|}  \Bigr)
\left| \Bigl(f_{z} + \rho |\mu| f_{z} \Bigr)^{\beta} \right| \leqslant 1,
 \end{equation}
with $\beta \in \C$ as in \eqref{uus7}. Since by \eqref{uus6} the complex numbers $\; p \, \rho(z) |\mu(z)|$ and $1 + \rho(z) |\mu(z)|$ have the same argument, the integrand, in fact, takes the equivalent form of \eqref{uus8}.

Concerning sharpness, let  $ f(z) = \frac{z}{|z|} |z|^{\frac{1-\eta}{1+\eta}}$ with $|\eta| =  k$. As  $\mu_f(z) = -\frac{z}{\bar z} \eta$,  the requirement $p \frac{\eta}{1+\eta} \geqslant 0$ determines the unimodular factor  
$\rho \equiv \eta/|\eta|$, i.e. $\eta=\rho |\mu|$. 
From $\re(\beta/p)=1$ one computes that $\; \re \frac{ \beta \eta}{1+\eta} = p \frac{\eta}{1+\eta}$. Since $f_z = \frac{1}{1+\eta}  |z|^{\frac{- 2\eta}{1+\eta}}$, a direct substitution shows that  the equality holds  in \eqref{uus8}.
\end{proof}

\subsection{Higher complex integrability}\label{subse:higher}

By Stoilow factorization,  for local integrability issues it is enough to control the behaviour of principal maps. Hence, our previous theorem yields  immediate corollaries.
\begin{theorem}\label{co:complexintegrability}
Suppose $f:\C \to \C$ is a $K$-quasiconformal mapping and $B = B(z,r) \subset \C$ is a disk. Then for any exponent $\beta \in \C$ such that
\begin{equation} \label{aito3}
 |\beta| + |\beta -2| <  2\cdot \frac{K+1}{K-1},
 \end{equation}
we have
\begin{eqnarray} \label{loc10}
 \hspace{-.8cm} c_1(K,\beta)  \left| \left(\frac{f(z + r) - f(z)}{r}\right)^\beta \right|
&  \leqslant & \frac{1}{|B|}  \int_B \left| f_{z}^\beta \right| \quad \leqslant   \\
   && \hspace{.8cm} c_2(K,\beta)  \left| \left(\frac{f(z + r) - f(z)}{r}\right)^\beta \right| \nonumber 
\end{eqnarray}
where the constants $c_1, c_2$ depend only on $K$ and $\beta$.
\end{theorem}
\begin{remark}The branches of $\; \log f_z(w)$, for $w \in B$, and of $\left(\frac{f(z + r) - f(z)}{r}\right)^\beta$ are here chosen  as in  \eqref{logs44} - \eqref{logs7}, using the {\it same} branch for both. More precisely, consider any of the branches of   $ \log \frac{f(z)-f(w)}{z-w}\,$ discovered in Proposition \ref{branch5}. As explained in \eqref{logs7}, this determines at a.e $z \in \C$ a branch of $\log \partial f(z)$, thus also the branches in \eqref{loc10}.  
On the other hand, the result holds for any choice of the branch.
\end{remark}
\begin{proof} By a change of variables and scaling, i.e. by using the auxiliary function
$$ F(w) = \frac{f(z+rw)-f(z)}{f(z+r)-f(z)}, \qquad w \in \C,
$$
 we may assume that $B = \DD$ and that $f$ fixes  the points $0$ and $1$.  With this scaling we are also reduced to  the case where the branch of  $\log \frac{f(z)}{z}$ is determined by the condition $\log f(1) = 0$. 

Use then the Stoilow factorization,
\begin{equation} \label{stoi}
 f(z)  = \varphi \circ f_0(z), 
 \end{equation}
 where $\varphi$ is conformal on $f_0(2\DD)$ and $f_0:\C \to \C$ is a principal quasiconformal mapping, with $\mu_{f_0}(z) = \mu_f(z)$ for $|z| < 2$ and  $\mu_{f_0}(z) = 0 $ for $|z| > 2$. In particular, 
 $$(f_0)_{\bar z}=\mu(z) \, (f_0)_z, \qquad |\mu(z)| \leqslant k\, \chi_{2\mathbb{D}}(z), \quad k \equiv \frac{K-1}{K+1} <1.
 $$

Considering first the inner factor, one applies Theorem \ref{main} to $f_0(2z)/2$, with complex parameter $1 \leqslant |p - 1| \leqslant \frac{1}{k} $. However, if we take $|p - 1| = \frac{1}{k}  $,Ê at points where $|\mu(z)| = k=\frac{K-1}{K+1} $, we have 
$p \, \rho k  = 1 + \rho k$ and  the integrand in  \eqref{uus8} vanishes, hence the estimate becomes useless for \eqref{loc10}.
 It is for this reason that we need to assume the strict inequality in \eqref{aito3}. 
 
Setting  $p=1+\frac{|\beta|+|\beta-2|}{2}$, we then have the strict inequalities 
$2< p < 1 + 1/k$, $k =  \frac{K-1}{K+1}$. Now $\rho(z) \equiv 1$ and \eqref{uus8}  gives
\begin{equation} \label{perusmuoto2}
 \frac{1}{\pi}\int_{B} 
\left| \Bigl(\partial f_{0}  \Bigr)^{\beta} \right| \leqslant  4 \, \frac{\max\{1,(1+k)^{1-\re \beta}\}}{1-k(p-1)} \;.
 \end{equation}
 
On the other hand,  \cite[(3.35)]{AIMb} shows for any quasisymmetric map $g:\Omega \to \Omega'$  the estimate
$$ |g(z) - g(z_0)| \leqslant \frac{c(\eta)}{r} \int_{\DD(z_0,r)} |\partial g|, \qquad   z \in \DD(z_0,r) \subset \Omega,
$$
where $\eta(t)$ is the modulus of quasisymmetry as in \eqref{qsymmetry}. 
 Further,  quasisymmetry  with Koebe distortion or
  \cite[(2.61)]{AIMb} gives  $\; \diam(f_0 B) \geqslant c_1(K)$. Therefore 
 $ \int_{B} \left| \partial f_{0}   \right|  \geqslant c(K)$.
 
Consequently, 
   $$ c(K)^2  \leqslant \left(\  \int_{B} 
\left| \Bigl(\partial f_{0}  \Bigr)^{\beta/2}  \Bigl(\partial f_{0}  \Bigr)^{1-\beta/2} \right|\right)^2
\leqslant  \int_{B} 
\left| \Bigl(\partial f_{0}  \Bigr)^{\beta}   \right|\,  \cdot \,  \int_{B} 
\left| \Bigl(\partial f_{0}  \Bigr)^{2-\beta}  \right|
  $$
  The requirement \eqref{aito3} holds for $\beta$ if and only if it does for $2-\beta$, and therefore
  \eqref{perusmuoto2} gives, too, the lower bound
  $$c(K,\beta) \leqslant   \int_{B} 
\left| \Bigl(\partial f_{0}  \Bigr)^{\beta}   \right|.
  $$

For the outer factor in the Stoilow factorization we apply Lemma  \ref{riemann}  in $\Omega = f_0(2\DD)$. This  gives 
$$ \left| \log \frac{\varphi(x)  -  \varphi(w)}{x-w} \, - \,  \log \varphi'(w) \right| \leqslant 10 \rho_\Omega(x,w), \qquad x,w \in \Omega.
$$
Choosing $x_0 = f_0(1)$, $w_0 =f_0(0)$ it follows that
$$  \left| \log \frac{\varphi(x_0)  -  \varphi(w_0)}{x_0-w_0} \right|  = \left| \log (f_0(1) -  f_0(0))\right| \leqslant (K-1) \log 8
$$
where the last estimate was shown in  \eqref{ysi}. Consequently, $\left|   \log \varphi'(w_0) \right| \leqslant C(K)$. To complete the argument note that as in the proof of Lemma \ref{riemann} the function  $|\log \varphi'(z)|$ is uniformly  Lipschitz with respect to the hyperbolic metric of $\Omega$. Since $f_0(\DD)$ has hyberbolic diameter in $\Omega$  
Êbounded in terms of $K$ only, Ê$|\log \varphi'(z)|\leqslant c(K) < \infty$ for $z \in f_0(\DD)$. With the chain rule and \eqref{stoi} we finally have 
$$ c_1(K,\beta)  \leqslant   \int_{B}  \left| \Bigl(\partial f  \Bigr)^{\beta} \right|  \leqslant c_2(K,\beta),$$
proving the claim.
\end{proof}

As for higher integrability with real exponents \cite[Section 13.4.1]{AIMb}, the theorem 
 can be interpreted in terms of the Muckenhoupt   $A_p$-weights.

\begin{corollary}\label{aapee} Suppose  $1 < p < \infty$ and $f:\C \to \C$ is a $K$-quasiconformal mapping. If   the  exponent $\beta \in \C$ satisfies
both  \eqref{aito3} and the dual condition 
\begin{equation} \label{aito7}
 |\beta| + |\beta +2(p-1)| <  2\cdot \frac{K+1}{K-1}(p-1),
 \end{equation} 
then $\left| f_{z}^\beta \right| \in A_p$.
  
Moreover, for each $1 < p < \infty$, outside this range of exponents the conclusion fails for some $K$-quasiconformal mapping $f:\C \to \C$.
\end{corollary}

As a last remark, the Stoilow factorization  works as well for mappings defined in proper subdomains of $\C$, and therefore arguing as in Theorem \ref{co:complexintegrability} we have the following local higher integrability stated as Theorem \ref{complexintegrability1} in the introduction. 


\begin{theorem}
\label{co:localcomplexintegrability}
Suppose  $f$ is a $K$-quasiconformal map on a domain $\Omega\subset\C.$ Then for any exponent $\beta \in \C$ in the critical ellipse
\begin{equation} \label{aito31}
 |\beta| + |\beta -2| <  2\cdot \frac{K+1}{K-1}.
 \end{equation}
we have 
\begin{equation}\label{loc}
\left| f_{z}^\beta \right| \in L^1_{loc}(\Omega).
\end{equation}
\end{theorem}

\noindent Sharpness of the previous result   is seen by testing with maps 
\refeq{eq:gammaalpha}. 

\begin{figure}
\includegraphics[width=9cm]{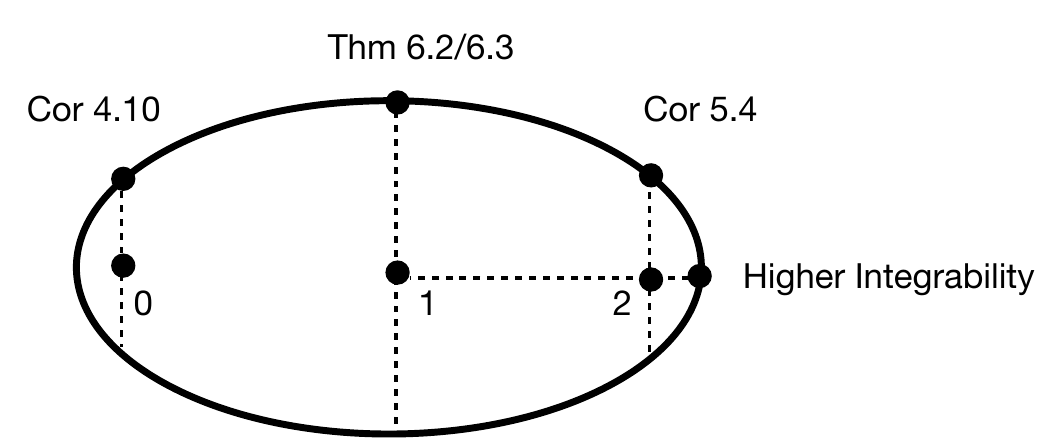}
\caption{The elliptical integrability region in Theorem \ref{co:localcomplexintegrability} and its various consequences.}
\label{fig:ellipseregion}
\end{figure}
\medskip

The previous theorem includes a number of special cases which are worth explicating, see Figure 
\ref{fig:ellipseregion}. The major axis corresponds to Higher Integrability of \cite{As}. Other boundary points on the ellipse uncover new phenomena. These will address exponential integrability of the argument for quasiconformal maps (Corollary \ref{cor:argumentintegrability}) as well for bilipschitz maps (Theorem \ref{thm:exparg}) and rotational multifractal spectrum (Corollary \ref{co:multifractal2} for quasiconformal maps and Theorem \ref{thm:bilipschitzspectrum} for bilipschitz maps). 

As a first special case with a purely imaginary exponent $\beta$ we obtain (cf.~Corollary \ref{cor:argumentintegrability-intro})

\begin{corollary}\label{cor:argumentintegrability} Suppose  $f$ is a $K$-quasiconformal map on a domain $\Omega\subset\C$. Then
$$
e^{b |\arg f_z|}\in L^1_{loc}\quad {\rm for \; all}\;\;{\rm positive }\;\;   b< \frac{4K}{K^2-1}.
$$
The result is optimal in the sense that it may fail with $b= \frac{4K}{K^2-1}$ for some $K$-quasiconformal $f$. Such an example is provided by {\rm\refeq{eq:gammaalpha}} with the choice 
$$ \tau = \frac12 \left(K+\frac{1}{K} \right) + \frac{i}{2} \left(K-\frac{1}{K} \right).
$$
\end{corollary}

\begin{remark}
We have chosen to derive Theorem \ref{co:localcomplexintegrability} from the precise weighted estimates of Theorem \ref{main}. However, to obtain the optimal exponent of integrability 
$\beta$ in Theorem \ref{co:localcomplexintegrability}  one may apply the interpolation lemma in many different ways. For instance, we could base the argument on the standard holomorphic flow given by $\mu_{\lambda}=\lambda \cdot \frac{\mu}{k}$ and the analytic family $f_z^\lambda$.
In order to obtain uniform $\mathscr{L}^2$-bounds in this setting, one needs to restrict the motion to $\{|\lambda|<1-\epsilon\}$ and use quasisymmetry.
\end{remark}

\smallskip

\section{Multifractal spectra}\label{se:multifractal}
\smallskip

The multifractal spectrum of a  Radon measure $\mu$ on $\R^n$ is (on the intuitive level) usually defined as the Hausdorff dimension of the set of points $x\in\R^n$ for which  $\mu (B(x,r))\sim r^\alpha$ for small radii. The rigorous definition has to be done carefully, and actually there are various notions of multifractal spectrum, see e.g. \cite{F} or \cite{O}. For a homeomorphism $f:\R^n\to\R^n$ the natural counterpart is the  multifractal spectrum of the induced push-forward measure $\mu = f_*(dx)$. 
In this spirit  the multifractal spectra of quasisymmetric maps of the real line was studied by the third author and Smirnov in \cite{PS}. Also closely related is Binder's work \cite{B} on the mixed integral means spectrum of conformal maps.

In this section we apply the results of the complex integrability of the gradient $f_z$ from the previous section, to
analyze the multifractal spectra  of  a $K$-quasiconformal  map $f \colon \C \to \C$. The complex integrability allows us to consider even the \emph{joint} multifractal behaviour with respect to both \emph{rotation \emph{and} stretching}.

According to the  discussion in Section \ref{infit}, fix $\alpha >0$ and $\gamma \in\R$ and  consider points 
$z\in\C$ with the following property:

\medskip

\noindent\hbox{}\qquad There is a decreasing sequence (depending on $z$) of  radii $(r_k)_{k\geqslant 1}$,\\
\noindent\hbox{}\qquad {with} $\displaystyle  r_k\rightarrow 0,$ such that \emph{simultaneously}
\begin{eqnarray}\label{eq:set}
&&\begin{split}
&\begin{cases}
\alpha=&\displaystyle \lim_{k\to\infty}\frac{\log |f(z+r_k)-f(z)|}{\log r_k}\\
&\\
\gamma= &\displaystyle  \lim_{k\to\infty}\frac{\arg (f(z+r_k)-f(z))}{\log |f(z+r_k)-f(z)|}\\
\end{cases}
\end{split}
\end{eqnarray}
\smallskip

One should  observe that  a single point can satisfy the condition \refeq{eq:set} 
for several different values of $(\alpha,\gamma ).$ In  \refeq{eq:set} the quantity $|f(z+r_k)-f(z)|$ measures  stretching in the direction of the positive real axis, but by the quasisymmetry of  $f$  any other fixed  direction gives the same result, and thus definition is  quite robust.

Introducing  the joint rotational and stretching multifractal spectrum for the class of all $K$-quasiconformal homeomorphisms, we are asking for  a characterisation of   the maximal size of the set of points  satisfying  \refeq{eq:set},
i.e. to determining the quantity
\begin{eqnarray}
\label{eq:jointspectrum}
\begin{split}
F_K(\alpha,\gamma)  := \sup \Big\{  \,&{\rm \dim_{\mathcal H} }(E)\;: \;  \refeq{eq:set}\;\; \textrm{holds for  every}\;\; z\in E, \; \;\textrm{for some} \\ \, & \mbox{$K$-quasiconformal mapping $f: \C\to\C$} \Big\}, 
\end{split}
\end{eqnarray}
where ${\rm \dim_{\mathcal H}}$ stands for the Hausdorff dimension. Note that, in view of Theorem \ref{localrotstret} there are no points satisfying \eqref{eq:set} unless 
$\tau= \alpha({1+i \gamma}) $ lies in the closed disk
\begin{equation}
\label{eq:hypdisk}
\Big| \; \tau\; - \; \,\frac12 \left(K+\frac{1}{K} \right) \Big| \leqslant \frac12 \Big(K-\frac{1}{K} \Big).
\end{equation}

The following theorem gives a complete description of the quasiconformal joint multifractal spectrum.
\begin{theorem}\label{th:multifractal} 
Assume that the parameters $\alpha >0$, $\gamma\in \R$ lie in the natural domain of definition of $F_K$, i.e. $\tau = \alpha({1+i \gamma})$ satisfies \eqref{eq:hypdisk}.
In this range the joint multifractal spectrum equals 
\begin{equation}\label{eq:quantity}
F_K(\alpha,\gamma) = (1+\alpha)-\frac{ \sqrt{(1-\alpha)^2(K+1)^2+ 4K \alpha^2\gamma^2}}{K-1}.
\end{equation}
\end{theorem}

Outside the range {\rm \eqref{eq:hypdisk}} we can set $F_K(\alpha,\gamma)=-\infty$, since then the set corresponding to $\tau = \alpha({1+i \gamma})$ is empty for any $K$-quasiconformal map. 

 \begin{remark}\label{re:cone}  It will be later  useful to observe that as a function  of the variable $\tau= \alpha({1+i \gamma})$ the function \refeq{eq:quantity} is determined  
 as the unique `cone'-like function on the closed disc
 \eqref{eq:hypdisk} with the properties:  the function takes 
the value $2$ at   $\tau=1$, vanishes  on the boundary of the disk  \eqref{eq:hypdisk} and is linear on each line segment that joins $1$ to the boundary circle. Figure \ref{fig:cone0} gives an illustration of  the graph of $F_K(\alpha,\gamma)$. 
\end{remark}

\begin{figure}
\includegraphics[width=5cm]{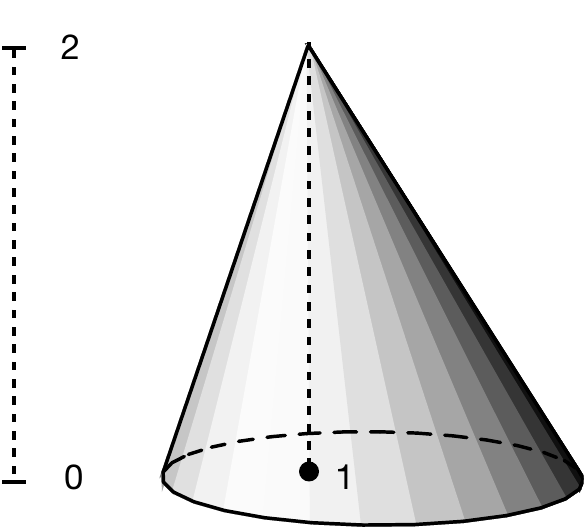}
\caption{The joint multifractal spectrum $F_K(\alpha,\gamma)$ as a function of the variable $\alpha(1+i \gamma)$.}
\label{fig:cone0}
\end{figure}


\noindent {\it Proof of Theorem \ref{th:multifractal}.} We begin with upper estimates for the spectrum $F_K(\alpha,\gamma)$.
Assume that $f:\C\to\C$ is $K$-quasiconformal and write for any $\alpha >0$ and $\gamma\in\R$
\begin{equation}\label{eq:Sf}
S_f(\alpha,\gamma):=\{ z\in\C\; :\; z\;\; \textrm{satisfies}\;\;\refeq{eq:set} \;\; \textrm{for some radii} \; r_k \to 0\}.
\end{equation}
Our task is to estimate  ${\rm \dim_{\mathcal H} }(S_f(\alpha,\gamma))$;
obviously, it is enough to estimate the Hausdorff dimension of   $S_f(\alpha,\gamma)\cap \DD$. 

Next, we apply the complex integrability of the gradient $f_z$ established in the previous section. For our present purposes  the most suitable form is given by Theorem \ref{co:complexintegrability}, which in particular  states for any exponent $\beta$ in the critical ellipse  $|\beta|+|\beta-2| < \frac{2}{k}$ and for
any disk $B(z,r)\subset B(0,2)$ that 
\begin{equation}\label{eq:int}
  \left| \left(\frac{f(z + r) - f(z)}{r}\right)^\beta \right|\;
\leqslant  \; c(K,\beta) \, r^{-2}  \int_B \left| f_{z}^\beta \right|. \quad  
\end{equation}
For consistency, recall  that above the argument of both $f_z$ and of the differences $f(z + r) - f(z)$
are obtained from a fixed choice of a  branch of the function \,
$\log\big((f(w)-f(z)/(w-z)\big)$,\, defined  for  $(w,z) \in \C^2$ with $w \neq z$;  the same branch  is used for every $z\in S_f(\alpha,\gamma)$ in the definition \eqref{eq:condition} below.

Assume then that $(\alpha,\gamma)$ satisfies \refeq{eq:hypdisk} with strict inequality, as it is clearly enough to consider this case.
Fix    arbitrarily small $\varepsilon\in (0,\alpha  ).$ By definition, we may select  for any $z\in S_f(\alpha,\gamma) \cap \DD$  a radius $r_z\in (0,\varepsilon )$ so that 
\begin{equation}\label{eq:condition}
\begin{split}
\frac{\log |f(z+r_z)-f(z)|}{\log r_z}&\in
 (\alpha-\varepsilon,\alpha+\varepsilon)\quad {\rm and}\\
\frac{\arg (f(z+r_z)-f(z))}{\log |f(z+r_z)-f(z)|}&\in
 (\gamma-\varepsilon,\gamma+\varepsilon).
\end{split}
\end{equation}
Vitali's covering lemma allows us to select countably many points
$z_n$ so that the discs $B_n:=B(z_n,r_n)$ with $r_n:=r_{z_n}$ are disjoint and $S_f(\alpha,\gamma) \cap \DD \subset \bigcup_{n}5B_n.$

We fix an arbitrary  $\beta$ from the open ellipse $|\beta|+|\beta-2| < \frac{2}{k}$
and observe that (\ref{eq:condition}) together with \eqref{eq:int} yields  for any of the discs
  $B(z_n,r_n)$  
  $$
r_n^{2+(\alpha-1)\re \beta-\alpha\gamma\im \beta + O(\varepsilon)}
\; =\; r_n^{2} \left| \left(\frac{f(z_n + r_n) - f(z_n)}{r}\right)^\beta \right|\; \leqslant \;\int_{B(z_n,r_n)} \left| f_{z}^\beta \right|.
$$
Above the exponent $O(\varepsilon )$ is uniform in $n.$
We thus obtain
\begin{eqnarray}\label{eq:basic}
&&\sum_n r_n^{2+ (\alpha-1) \, \re  \beta -\alpha\gamma \,  \im  \beta  +O(\varepsilon )} \; \leqslant\; \sum_n  c(\beta)\int_{B(z_n,r_n)} |(f_z)^{\beta}|\\ & &\leqslant c(\beta)\int_{2\DD}  |(f_z)^{\beta}|\nonumber <\infty\nonumber
 .
\end{eqnarray}
Since $S_f(\alpha,\gamma)\cap \DD \subset \cup_{n}5B_n$, and $\varepsilon >0$ can be taken arbitrarily small,
it immediately follows  that  $\,2  +(\alpha-1) \re  \beta -\alpha\gamma\, \im \beta\;$ is  an upper bound for ${\rm \dim_{\mathcal H} }(S_f(\alpha,\gamma))$, for any $\beta$ in the critical ellipse.  As $f$ was an arbitrary $K$-quasiconformal map we infer that
\begin{equation}\label{eq:beta}
 F_K (\alpha,\gamma)\leqslant \inf_{\beta } \{2  +(\alpha-1) \re  \beta -\alpha\gamma\, \im \beta\},
\end{equation}
where the supremum is taken over the set $|\beta|+|\beta-2| < \frac{2}{k}$.
Using the parametrization of the ellipse  in condition \refeq{eq:repre} we get equivalently
\begin{eqnarray*}
F_K(\alpha,\gamma)&\leqslant& \min_{\theta\in [0,2\pi)} \left\{ (1+\alpha) + \frac{\alpha-1}{k} \cos \theta - \alpha\gamma \,\frac{\sqrt{1-k^2}}{k} \sin \theta \right\}\\ &=& (1+\alpha)-\frac{1}{k} \sqrt{(1-\alpha)^2+(1-k^2)\alpha^2\gamma^2}\\
&&=(1+\alpha)-\frac{ \sqrt{(1-\alpha)^2(K+1)^2+ 4K \alpha^2\gamma^2}}{K-1},
\end{eqnarray*}
as $k = \frac{K-1}{K+1}$.

In fact, as soon as we have the equality in \refeq{eq:beta}, from this representation one most directly sees the cone like property of $F_K(\alpha,\gamma)$, as discussed in Remark \ref{re:cone}.

It remains to find  lower bounds for $F_K(\alpha,\gamma)$, and for this  we  provide examples verifying the optimality of the estimate \refeq{eq:beta}. The examples  are constructed by iterating the map \refeq{model12} in a self-similar manner inside  interlaced annuli that form a Cantor like structure. The quasiconformal map spirals only inside the annuli, elsewhere it is a similarity. The reader may compare the construction   with that in  \cite[Thm. 13.6.1]{AIMb}). 

Again, fix  $K>1$ and a pair $(\alpha,\gamma)\in (0,\infty)\times\R$ such that $\tau = \alpha(1+i\gamma)$ satisfies  (\ref{eq:hypdisk}) with a strict inequality. This initial  knowledge 
allows us to fix the auxiliary parameters $t$ and parameters 
 $(\alpha_0,\gamma_0)$ via declaring  $t$ to be  \emph{the smallest positive number such that}
\begin{equation}\label{apu8}
\begin{split}
t^{-1}\big(\alpha({1+i\gamma})-1\big)+1=:\alpha_0({1+i\gamma_0})\in \overline{B}_K,
\end{split}
\end{equation}
where 
$$\overline{B}_K\;:\; =\;\Bigg\{ \tau \in \C\; :\; \Big| \; \tau\; - \; \,\frac12 \left(K+\frac{1}{K} \right) \Big| \leqslant \frac12 \Big(K-\frac{1}{K} \Big)\Bigg\}.
$$
Clearly $t\in (0,1).$
There will be an additional  free parameter  $r\in (0,1/8)$  that we fix so small that $r<2^{-4}r^t$, and denote 
\begin{eqnarray}\label{s}
s:=r^{{\gamma}{\alpha}/{\alpha_0}{\gamma_0}}=r^t.
\end{eqnarray}

We next  construct  a $K$-quasiconformal map
$\phi$ (that depends on the parameters $\alpha_0,\gamma_0$ and $r$) by suitably iterating    the model map \refeq{eq:gammaalpha} with parameters $(\alpha_0,\gamma_0).$
  For that end we select  first inductively an infinite collection 
 of annuli, partitioned into levels  $j\geqslant 0$, in such a way that the level  $j$ annuli will have  outer radius $r^j$ and inner radius $sr^j.$  In level 0 there is only one annulus, i.e. $B(0,1)\setminus \overline{B(0,s)}.$ Assume that we have already constructed all the annuli up to level $j-1$, where $j\geqslant 1.$ For each annulus of the level $j-1$, say for $B(z_0,r^{j-1})\setminus\overline{B(z_0,sr^{j-1})}$, we may pick $N:=\big(\lfloor (s/2r)\rfloor\big)^2 =\big(\lfloor r^{(t-1)}/2\rfloor\big)^2\geqslant r^{2(t-1)}/8$  disjoint discs 
$B(z_m,r^j)$, with $m=1,2,\ldots, N$, all lying inside the punctured disc $B(z_0,sr^{j-1})\setminus\{ z_0\}$.
The corresponding level $j$ annuli are $B(z_m,r^{j})\setminus\overline{B(z_m,sr^{j})}$ for $m=1,\ldots ,N.$ By performing this operation for all annuli of the level $j-1$ we obtain the complete collection of annuli of level $j$, and their total number is $N^j$. This collection can be written as  $\{  B\setminus \overline{sB}\; :\; B\in {\mathcal A}_j\}$, where  ${\mathcal A}_j$ stands for the set  of all the outer discs of the level $j$ annuli. Obviously this construction can be made via suitable similitudes so that the intersection
\begin{equation}\label{eq:E}
E:=\bigcap_{j=1}\big(\bigcup_{B\in{\mathcal A}_j}\overline{B} \, \big)
\end{equation}
becomes a self-similar Cantor set in the plane.

Given an arbitrary annulus $B\setminus \overline{sB}$, with $B=B(w,R)$ we define the corresponding rotation map $\psi_{B}$ by setting
\vspace{-.4cm}

\begin{equation}\label{apu2}
\psi_{B}(z)=\left\{
\begin{array}{ll}
z& {\rm if}\;\; z\not\in B\\
w+ R\frac{(z-w)}{|z-w|}\left(\frac{|z-w|}{R}\right)^{\alpha_0(1+i\gamma_0)}&{\rm if}\;\; z\in B\setminus sB\\ 
{\rm continuous\;  similarity\; extension} &{\rm if}\;\; z\in sB.
\end{array}
\right.
\end{equation}
\vspace{-.25cm}

The quasiconformal mapping $\phi$  is then defined  via
an  inductive construction, see Figure \ref{fig:cantor} for illustration.
First set $\phi_0(z)=z$ and assume that $\phi_{j-1}$ is already defined. Then choose
\begin{equation}\label{apu3}
\phi_j(z)=\left\{
\begin{array}{ll}
\phi_{j-1}(z)& {\rm if}\;\; z\in \C\setminus (\bigcup_{B\in \mathcal{A}_j}B)\\
\psi_{\phi_{j-1}(B)}(\phi_{j-1}(z))&{\rm if}\;\; z\in B\;\;{\rm with }\;\; B\in {\mathcal A}_j.
\end{array}
\right.
\end{equation}
The new spiralling introduced via $\phi_j$ takes place in the set where $\phi_{j-1}$ is conformal (actually even complex linear). Hence, the fact $\alpha_0(1+i\gamma_0) \in \overline{B}_K $ implies that each $\phi_j$ is $K$-quasiconformal, and so is our final map $\phi$, where 
\begin{equation}\label{eq:phi}
\phi (z):=\lim_{j\to\infty} \phi_j(z).
\end{equation}

\begin{figure}
\includegraphics[width=4cm]{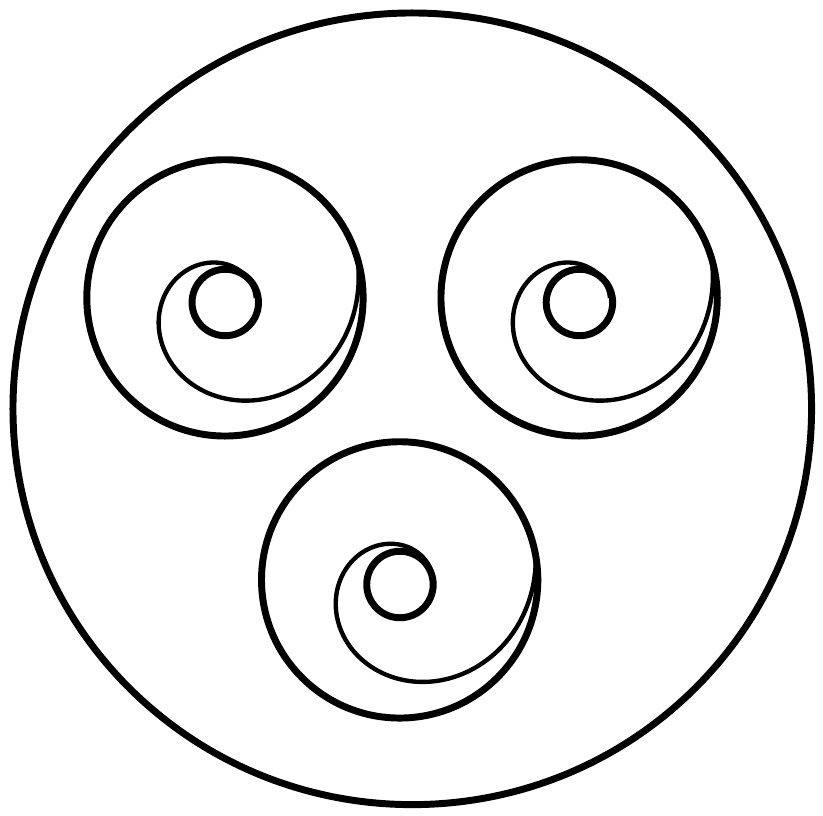}
\caption{Illustration of the mapping $\phi$ constructed in \eqref{eq:phi}}
\label{fig:cantor}
\end{figure}

Let us then consider how $\phi$ maps fixed level $j\geqslant 1$ balls. Any  ball $B\in {\mathcal A}_j$  has radius $r_j:=r^j$ and it is mapped into a ball $B'=\phi (B)$ with radius $r'_j$,
where the definition of $\phi$ together with \refeq{s} and  \refeq{apu8} yields that
\begin{equation}\label{apu4}
r'_j= { s}^{j(\alpha_0-1)}r_j=r^{jt(\alpha_0-1)}=r^{j(\alpha-1)}=
( r_j)^{\alpha}.
\end{equation}
Hence the stretching of $\phi$ has the desired order with respect to the center point of $B$. Moreover, this is also true for the rotation of $\phi$ since if $B=B(z,r^j)$ we obtain  directly by construction and using normalization \eqref{branch} that
\begin{equation}\label{apu6}
\arg  (f(z+r^j)-f(z))=  {j\gamma_0\alpha_0}\log s=\gamma \log r '_j .
\end{equation}

Every point in $E$  is inside a disc of level $j$. It is not quite the center point, but at a finite hyperbolic distance from it. Thus the above observations together with the robustness of the definition of pointwise rates of rotation and stretching (based on Theorem \ref{ptwise}), allow us to conclude that
$$E\subset S_\phi(\alpha,\gamma ).$$

On the other hand,
the Hausdorff dimension $\tau:={\rm dim}_{\mathcal H}(E)$ is computed from the
equation
$
Nr^\tau=1,
$
and by recalling that $N=\big(\lfloor r^{(t-1)}/2\rfloor\big)^2$ we obtain in the limit $r\to 0^+$
that $\tau\to 2(1-t).$ Hence
$$F_K(\alpha,\gamma)\geqslant 2 (1 -t).$$
According to definition \eqref{apu8} and the cone-type characterization observed in Remark \ref{re:cone}, this exactly means that $F_K(\alpha,\gamma)$
has the right lower bound, and  the proof of Theorem \ref{th:multifractal} is complete. \hfill $\Box$
\medskip

Combining  general properties of holomorphic motions with the above quasiconformal multifractal bounds quickly gives
\smallskip

\noindent {\it Proof of Theorem \ref{holodim}}.  Suppose $\Psi: \DD \times E \to \C$ is a holomorphic motion of a set $E \subset \C$. By Slodkowski's generalized $\lambda$-lemma \cite{Slod}  $\Psi$ extends to a motion of the whole complex plane, while the original $\lambda$-lemma of Ma{\~n}{\'e}, Sad and Sullivan \cite{MSS} proves  the extended map $\Psi_\lambda(z) = \Psi(\lambda,z)$ to be quasiconformal in $\C$. A Schwarz-lemma type argument shows that 
$$ K(\Psi_\lambda) \leqslant \frac{1+|\lambda|}{1-|\lambda|}, \qquad \lambda \in \DD, 
$$  
for details see e.g. \cite[pp. 303-304]{AIMb}. The dimension bounds \eqref{dimbound2}  hence follow  from 
Theorem \ref{th:multifractal}.
\hfill $\Box$

\medskip

We next consider  the upper and lower stretching exponents of
$f$ at point $z$ defined by 
\begin{eqnarray*}
&&\overline{\alpha}_f(z)= \limsup_{r\to 0^+} \frac{\log |f(z+r)-f(z)|}{\log r} \; , \label{yksi}  
\\ &&\underline {\alpha}_f(z)= \liminf_{r\to 0^+} \frac{\log |f(z+r)-f(z)|}{\log r}\, .
\end{eqnarray*}
In a similar manner, the upper and lower rates of rotation are given by
\begin{eqnarray}
&&\overline{\gamma}_f(z)=  \limsup_{r\to 0^+} \frac{\arg (f(z+r)-f(z))}{\log |f(z+r)-f(z)|},
\\&&\underline {\gamma}_f(z)= \liminf_{r\to 0^+}\frac{\arg (f(z+r)-f(z))}{\log |f(z+r)-f(z)|}.
\end{eqnarray}

In the following result the  novelty is the estimate \eqref{eq:co2} 
for the Hausdorff dimension of  the set  
 where each point has a prescribed rotation index. This estimate also shows that although our method  considers  simultaneous rotation and stretching, it is capable of producing optimal estimates for the pure rotational multifractal spectra.
\begin{corollary}\label{co:multifractal1}
Let $f:\C\to \C$ be a  $K$-quasiconformal map and $k = \frac{K-1}{K+1}$. Then  
\begin{equation}\label{eq:co1}
\dim_{\mathcal H}\Big( \{z\; :\; \overline{\alpha}_f(z)=\alpha \quad {\rm or } \quad \underline{\alpha}_f(z)=\alpha \} \Big)
  \leqslant
1+\alpha-\frac{1}{k}|1-\alpha|,    
\end{equation}
for any $\alpha \in  [K^{-1},K].$  Moreover,
\smallskip
\begin{equation}\label{eq:co2}
\dim_{\mathcal H}\Big(\{ z\; :\; \overline{\gamma}_f(z)=\gamma \quad {\rm or } \quad
\underline{\gamma}_f(z)=\gamma \} \Big)
  \leqslant
 2-\frac{k^{-1}-k}{\sqrt{1+\gamma^{-2}}-k}
\end{equation}
for any $\gamma$ with $ |\gamma|\leqslant {(K-K^{-1})}/2$.

If either $\alpha$ or $\gamma$ lies outside the given interval, then there are no points with exponent $\alpha$ or index $\gamma$, respectively.
Both estimates are optimal.
\end{corollary}
\begin{proof}
The first statement \refeq{eq:co1} is  deduced by a small modification of the proof of Theorem
\ref{th:multifractal}. For each $z\in \overline{\alpha}_f(z)$ (resp. 
$z\in \underline{\alpha}_f(z)$) one chooses a radius $r_z$ that satisfies  the first condition in \eqref{eq:condition} and, in the later stage of the proof, employs only real  parameters $\beta$ in the allowed range
$\beta\in (1-1/k,1+1/k).$ In view of \refeq{eq:basic} and \refeq{eq:beta} the  Hausdorff dimension  of the left hand side set in bounded by
$$
\inf_{\beta\in (1-1/k,1+1/k)}\{ 2+(\alpha-1)\beta\},
$$
and the claim follows. That the sets in question  are empty for values $\alpha\not\in [K^{-1},K]$ follows immediately from Theorem \ref{localrotstret}, and the optimality is obtained via considering maps \refeq{eq:phi} with
$\gamma=0.$

Towards the second statement , we first fix $\gamma>0$ and write 
\begin{equation}\label{eq:E2}
E:=\{ z\;:\;\overline{\gamma}(z)=\gamma\;\; {\rm or}\;\;\underline{\gamma}(z)=\gamma\}.
\end{equation} We claim that
$E\subset \bigcup_{\alpha\in I}S_f(\alpha,\gamma),$ where  $I$ is the interval
of allowed $\alpha$, i.e. the values of $\alpha$ so that the pair $(\alpha,\gamma)$ satisfies \refeq{eq:hypdisk}.  
Namely, if e.g. $\overline{\gamma}(z)=\gamma,$ we may pick a sequence $r_k\to 0$ such that the second condition in \refeq{eq:set} is satisfied, and by moving to a further subsequence the first condition \refeq{eq:set} holds as well, with some allowed value for $\alpha .$
One observes that the proof of Theorem  \ref{th:multifractal}  is quite robust and yields immediately for any $(\alpha',\gamma)$ satisfying \refeq{eq:hypdisk} and for any  $\varepsilon >0$ the estimate
$$
\dim_{\mathcal H}\Big(\bigcup_{\alpha'-\varepsilon \leqslant \alpha\leqslant 
\alpha'+\varepsilon }S_f(\alpha,\gamma)\Big)\leqslant F_K(\alpha',\gamma)+c\varepsilon
$$
with a uniform constant $c=c(\gamma, K)$.
 By covering the interval $I$ with   finitely many intervals 
$(\alpha'-\varepsilon \leqslant \alpha\leqslant 
\alpha'+\varepsilon) $ and since $\varepsilon >0$ is arbitrary we deduce that
$$
\dim_{\mathcal H}(E)\leqslant \sup_{\alpha'\in I} F_K(\alpha',\gamma).
$$
This estimate is obviously optimal in view of the example \eqref{eq:phi}. 

In order to compute the above supremum
one may shorten computations by recalling  from Remark \ref{re:cone} that $F_K(\alpha,\gamma)= 2(1-t),$
with $\alpha(1+i\gamma)-1=t(\alpha_0(1+i\gamma_0)-1),$ where $\alpha_0(1+i\gamma_0)$ is a boundary point of the disc \refeq{eq:hypdisk}.
We parametrize the boundary by  $\alpha_0(1+i\gamma_0)
=A+a\cos \theta +ia\sin \theta$ where $A:=(K+1/K)/2$ and $a:=(K-1/K)/2$. Then
$t$ is determined from the condition
$$
\gamma =\frac{t\alpha_0\gamma_0}{1+t(\alpha_0-1)}=
\frac{ta\sin \theta}{1-t+t(A+a\cos \theta)}
$$
and we get  $t= \big(1+(a/\gamma)\sin\theta-a\cos\theta -A\big)^{-1}$.
The minimal value of this quantity is obviously
$$t_{min}:=
\frac{1}{1-A+a\sqrt{1+\gamma^{-2}}},
$$
which yields  the stated dimension bound.

The emptiness of the sets under consideration in the case where the parameters lie outside the stated ranges follows  again from Theorem \ref{ptwise}. Here one notes that the maximal slope for lines through origin that intersect the closed disc defined by condition  \refeq{eq:hypdisk} equals 	$(K-K^{-1})/2.$
\end{proof}

One may also bound the size of the image of the set with prescribed rotation rate as follows:
\begin{corollary}\label{co:multifractal2}
For any $K$-quasiconformal map $f:\C\to\C$ one has
\smallskip
\begin{equation}\label{eq:co3}
\begin{split}
\dim_{\mathcal H}\Big(f\{ z\; :\; \overline{\gamma}_f(z)=\gamma \; \; {\rm or } \;\;
\underline{\gamma}_f(z)=\gamma \}\Big)
\leqslant \;\; 
2-  \frac{4K}{K^2-1}\, |\gamma| 
\end{split}\nonumber
\end{equation}
for  any $\gamma$ with $ |\gamma|\leqslant (K-K^{-1})/2$. 

For  $\gamma$ outside the interval there are no points with this rotation rate. The result is optimal.
\end{corollary}

\begin{proof} We again modify the proof of Theorem \ref{th:multifractal}.
Choose for each point $z\in E
$  (where $E$ was defined in \refeq{eq:E2}) a radius $r_z$ so that the second condition in \refeq{eq:condition} holds with a fixed $\varepsilon>0$, and
pick the disjoint balls $B_n=B(z_n,r_n)$ as before. Define $\alpha_n$ through
$r_n^{\alpha_n}=|f(z_n+r_n)-f(z)|=:r'_n.$ In \refeq{eq:basic} and \refeq{eq:beta} apply  only the values $\beta =2+ib$ with 
$|\beta-2|+|\beta |<2/k$, or equivalently $|b|<k^{-1}-k$. In this situation \refeq{eq:basic} can be written
in terms of the radii $r'_n$  in the form $\sum_n {r'}_n^{2-\gamma b+O(\varepsilon)}\leqslant C$. By quasisymmetry,  the balls $cB(f(z_n),r'_n)$ cover
the image $f(E)$, where $c$ depends only on $K.$ Hence the analogue of 
\eqref{eq:beta} takes the form
$$
\dim_{\mathcal H}(f(E))\leqslant \inf_{|b|<(k^{-1}-k)}\big (2-\gamma b\big)=2-|\gamma | (k^{-1}-k),
$$
which yields the stated estimate.

Optimality is verified by considering the map $\phi$ from \refeq{eq:phi} with the stretching and rotation indices $(\alpha (\gamma ),\gamma),$ with
$|\gamma|< (K-K^{-1})/2$, and where the judiciously chosen value of $\alpha$
equals 
$$
\alpha (\gamma ):=\frac{1}{1+|\gamma|k}.
$$
One also notes that $\dim (\phi (E))=\dim (E)/\alpha$ as one easily computes by observing that $\phi (E)$ is likewise self-similar.
It is of interest to observe that, independently of the value of $\gamma>0$, the half-line determined by the points $1$ and $\alpha(\gamma )(1+i\gamma)$ intersects $\partial \overline{B}_K$ at the point 
 $\frac{1-k^2}{1+k^2}+\frac{i2k}{1+k^2}$.
\end{proof}

\begin{remark}{\rm
One could of course also ask for bounds of the Hausdorff dimension of the image
in the contexts of  Theorem \ref{th:multifractal} or of the first part of Corollary \ref{co:multifractal1}. In these cases   the optimal bound is obtained by multiplying the already obtained results by $\alpha^{-1}$, since  one may  cover the image by  balls $B(z_j,cr_j^{\alpha-\varepsilon})$,
where the disks $B(z_j,r_j)$ are as in the respective proofs.  }
\end{remark}

It is also reasonable to study   Minkowski type multifractal spectra by  considering dimension estimates  of the type
 \[ D_{f,M}(\alpha,\gamma) := \lim_{\varepsilon \to 0} \limsup_{{r} \to 0} \frac{\log N({r},\alpha,\gamma,\varepsilon)}{|\log {r}|},
\]
where $N({r},\alpha,\gamma,\varepsilon)$ is a maximum number of disjoint disks $B_n=B(z_n,{r})$
with center $z_n \in \mathbb{D}$, such that the stretching and  rotation of $f$ on $B_n$ satisfy \refeq{eq:condition} (with $r_z$ replaced by $r$). The  proof of Theorem  \ref{th:multifractal} applies with mere cosmetic changes and yields 
\begin{corollary}\label{co:multifractal3}
Let $f:\C\to \C$ be a  $K$-quasiconformal map. Then  

$$
D_{f,M}(\alpha,\gamma)\leqslant 1+\alpha -\frac{ \sqrt{(1-\alpha)^2(K+1)^2+ 4K \alpha^2\gamma^2}}{K-1},
$$
\vspace{-.3cm}

\noindent whenever $\tau =  \alpha({1+i \gamma}) $ lies in the disk \eqref{eq:hypdisk}. The estimate is the best possible. 
\end{corollary}

\begin{remark}\label{re:domainmultifractal}{\rm By standard localization and Stoilow factorisation,  the results in this section generalize for $K$-quasiconformal maps between arbitrary domains.
}
\end{remark}

\begin{remark}\label{re:imageside}{\rm Obviously Theorem \ref{th:multifractal} can be equally well understood as a statement on the combined ordinary multifractal spectrum of the pull-back measure  $\mu:=f^*(dx)$ and the rotation spectrum.
}
\end{remark}

\section{Bilipschitz maps and rotation}\label{se:bi-Lipschitz}

In this section, we apply our quasiconformal methods to bilipschitz maps. These have, by the very definition, trivial stretching but may exhibit non-trivial rotation. On the other hand, extremal quasiconformal mappings for various ``stretching problems'' often exhibit no rotation at all. This  leads to a number of dual analogies, as described in  Table \ref{table:qcvsbilip}  in the introduction. Below we uncover  the various entries from this dictionary.


As discussed in Section \ref{bilipqc}, orientation preserving $L$-bilipschitz mappings  $f \colon \Omega \to \mathbb{C}$ are $L^2$-quasiconformal (and if $f$ is not orientation preserving, then its conjugate $\overline{f}$ is). 
%
An archetypical example of a bilipschitz map with spiralling behaviour is the logarithmic spiral map
\begin{equation} 
\label{eq:spiralmap}
s_\gamma(z)=z|z|^{i \gamma}, \quad \gamma \in \mathbb{R},
\end{equation}
that we already discussed before.
Indeed, it is not hard to verify that $s_\gamma \colon \mathbb{D} \to \mathbb{D}$ is $L$-bilipschitz where $L \ge 1$ satisfies
\[ L-\frac{1}{L}=|\gamma|.
\]
In the following, we study rotational behaviour of general planar bilipschitz mappings in various points of views.
All of our results will be consequences of the quasiconformal theory. 
Nevertheless, in this way we obtain sharp results even in the bilipschitz category. One way to explain this phenomenon is  that often the maps  extremal in the bilipschitz category,  such as the logarithmic spiral map (and its iterated variants), are not only bilipschitz but also area-preserving; thus there is an exact correspondence between the optimal quasiconformal and bilipschitz constants of the form $K=L^2$.

\subsection{John's problem} 
\label{subsection:John}
Denote by $A=A(r,R)=\{z \in \C : r<|z|<R \}$ the annulus with radii $0<r<R<\infty$. Let $f \colon A \to A$ be an $L$-bilipschitz map  identity  $f(z)=z$ on the outer boundary $|z|=R$, with  a prescribed rotation (parametrized by $\gamma \in \R$) on the inner boundary, 
\[ f(z)=z e^{i \gamma \log(R/r)}, \quad |z|=r.
\]
John's problem asks for constraints on the bilipschitz constant $L$ for such a map to exist. Note that the analogous problem of finding quasiconformal maps between annuli (of different conformal modulus) is attributed to Gr\"otzsch.
In his work \cite{Joh} F. John  obtained some quantitative results in the asymptotic regime $L \to 1$, while the complete solution was given by  Gutlyanski{\u\i} and Martio \cite{GM}: the rotation parameter $\gamma$ and the bilipschitz constant $L$ need to satisfy
\begin{equation}
\label{eq:gm}
|\gamma| \leqslant L-\frac{1}{L}.
\end{equation}
As discussed above, the spiral maps such as in  \eqref{eq:spiralmap} provide extremal examples.
For related recent results in the class of mappings of finite distortion, see \cite{BFP}. 

We begin by  generalizing 
\eqref{eq:gm} to all  $L$-bilipschitz maps between arbitrary domains.

\subsection{Pointwise rotation}  
Recall the upper and lower rates of rotation from Section \ref{se:multifractal}, definitions  \eqref{yksi}. For bilipschitz mappings it is equivalent to use  the following convenient formulations 

\begin{eqnarray*}
&&\overline{\gamma}_f(z)=  \limsup_{r\to 0^+} \, \frac{\arg (f(z+r)-f(z))}{ \log r},
\\
&&\\
&&\underline {\gamma}_f(z)= \liminf_{r\to 0^+} \, \frac{\arg (f(z+r)-f(z))}{\log r }.\\
\end{eqnarray*}
\begin{proposition}
\label{thm:pointwise}
Let $f \colon \Omega \to \Omega'$ be an $L$-bilipschitz homeomorphism between planar domains. 
Then the  rates of rotation satisfy the following pointwise bound
\[ |\overline{\gamma}_f(z)|,|\underline{\gamma}_f(z)| \leqslant L-\frac{1}{L}, \quad z\in \Omega.
\]
The spiral map \eqref{eq:spiralmap} shows that  this is best possible in general.
\end{proposition}
\begin{proof} After taking the conjugate if necessary, the $L$-bilipschitz map becomes $L^2$-quasiconformal.  We 
observe that the bilipschitz property of $f$ ensures that the limit \eqref{alffa} is equal to one at every point regardless of the subsequence of radii chosen. Hence an application  of Theorem \ref{localrotstret} yields for (any) rate of rotation $\gamma$ the inequality
$$
 \left| i\gamma \, - \,\frac12 \left(K+\frac{1}{K} -2\right)  \right| \leqslant \frac12 \left(K-\frac{1}{K} \right).
$$
As  $K\leqslant L^2$, a simplification yields the stated bound.
\end{proof}

The slightly weaker estimate  $|\gamma| \leqslant \sqrt{L^2-1}$ was derived by Freedman and He \cite{FH} in connection with  the bilipschitz factoring problem for  the map $s_\gamma$. We will see in Section \ref{subsection:factoring} how the optimal form of Theorem \ref{thm:pointwise} leads to an optimal answer of the factoring problem, as well.  

Theorem \ref{thm:pointwise} is sharp as a pointwise estimate, but one expects that extremal spiralling behaviour cannot simultaneously occur at many places. This is the theme we discuss next. 

\subsection{BMO and Exponential integrability} On a historical side, 
 the function space of bounded mean oscillation was originally introduced by John and Nirenberg  \cite{Joh,JN} exactly in the context of rotational phenomena of bilipschitz maps, i.e. John's problem above. Roughly speaking, John established a discrete variant of the fact, c.f.  Proposition \ref{prop:hamilton}, that
\[ \arg f_z \in BMO.\]
This coupled with the John-Nirenberg lemma leads then to bounds on the rotation problem considered in section \ref{subsection:John}. 

In view of the previous discussion it is natural to ask for the best possible exponential integrability of the argument $\arg f_z$. As an application of Theorem \ref{co:localcomplexintegrability} we obtain immediately

\begin{theorem}
\label{thm:exparg}
Let $f \colon \Omega \to \C$ be an $L$-bilipschitz mapping. Then for any $0 \leqslant b < \frac{2L}{L^2-1}$ we have
\begin{equation} \label{luku6} \exp(b |\arg f_z|) \in L^1_{{\rm loc}}(\Omega).
\end{equation}
Furthermore, the integrability  fails  at the borderline $b= \frac{2L}{L^2-1}$ for some $L$-bilipschitz mapping $f$.
\end{theorem}
\begin{proof}  
 Since  $|f_z|$ is uniformly bounded above and below,   only the imaginary part of $\beta$ plays now a role in  Theorem \ref{co:localcomplexintegrability}. But for $\beta$'s
    inside  the critical ellipse $|\beta| +|\beta-2| < 2\frac{L^2+1}{L^2-1}$,  
  the supremum of the  imaginary part  equals $\frac{2L}{L^2-1}$ (and infimum equals $-\frac{2L}{L^2-1}$).
The optimality is seen by considering the model map \refeq{eq:spiralmap} with $L-1/L=|\gamma |.$
\end{proof}

\subsection{Multifractal spectrum} Interpreting
Theorem \ref{thm:exparg} more geometrically we use it  for  multifractal bounds on rotation.
Indeed, by specializing Theorem \ref{th:multifractal} to the bilipschitz case, that is, by setting $\alpha=1$ and $K=L^2$,  we obtain the following characterisation 

\begin{theorem}
\label{thm:bilipschitzspectrum}
Let $f \colon \Omega \to \Omega'$ be an $L$-bilipschitz mapping, where  $\Omega$ and $\Omega'$ are  planar domains. Then we have the optimal bounds
$$
\hspace{-.2cm}\dim_H \big(\, \{ z \in \Omega : \; \overline{\gamma}_f(z)=\gamma   \;\; \, {\rm or } \;\;\,
  \underline{\gamma}_f(z)=\gamma \}\, \big) \;   \leqslant \;  \; 2-\frac{2L}{L^2-1} |\gamma|
$$
for every admissible $ |\gamma| \leqslant L-\frac{1}{L}$. 
\end{theorem}
\smallskip

Theorem \ref{1.1} is an immediate corollary.
The optimality of the bounds,  i.e. that they hold as an equality for some $L$-bilipschitz mapping,  follows from the construction in Theorem \ref{th:multifractal} with $\alpha=1$.
Figure \ref{fig:spectra} above contrasts Corollary \ref{co:multifractal1} and \eqref{eq:co1} with Theorem \ref{thm:bilipschitzspectrum}. Observe also  that for bilipschitz maps the dimension of any set is preserved, whence the above estimate is valid also for the image of the set of prescribed rate of rotation. 

\begin{figure}
\includegraphics[width=12cm]{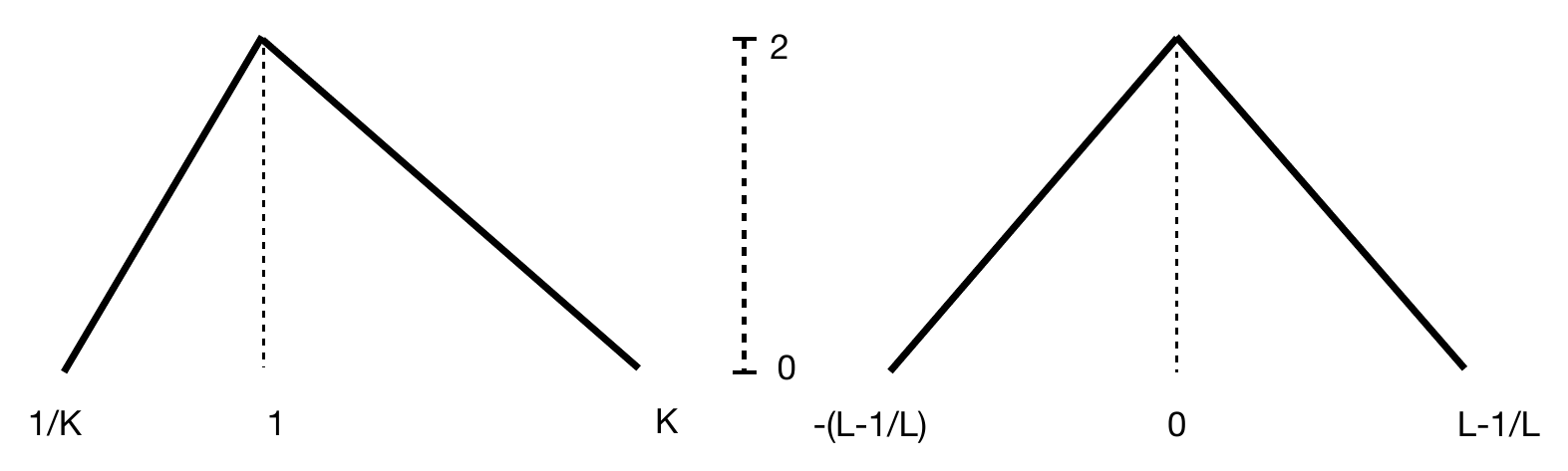}
\caption{Multifractal spectra: quasiconformal stretching (left) and bilipschitz rotation (right).}
\label{fig:spectra}
\end{figure}

\subsection{Factoring the logarithmic spiral} 
\label{subsection:factoring} A basic open question in the study of  bilipschitz mappings (in $\R^n$)
 is whether such a map can be represented as a composition of $(1+\varepsilon)$-bilipschitz mappings, for any $\varepsilon >0$. The factoring is known only in dimension $n = 1$;  see \cite{FM} for  recent general results on this theme.
 Towards this problem Freedman and He \cite{FH} studied the factoring of  the logarithmic spiral map \eqref{eq:spiralmap}. 
As an application of our pointwise rotation estimates we will revisit their question. 

One way to factor $s_\gamma$ to maps of small bilipschitz distortion is to simply write it as a composition of slower spirals:
\[ s_\gamma = \underbrace{s_{\gamma_0} \circ s_{\gamma_0} \circ \ldots  \circ s_{\gamma_0}}_{\mbox{ \tiny $N$ terms}},
\]
where $\gamma_0 = \gamma /N$. The next Theorem says that this is the most efficient way.

\begin{theorem}
\label{thm:factoring}
Let $s_\gamma \colon \bar{\mathbb{D}} \to \bar{\mathbb{D}}$ be factored as
$s_\gamma = f_N \circ f_{N-1} \circ \ldots f_1$, where each $f_i$ is an $L_0$-bilipschitz map of a closed Jordan domain in $\R^2$, $L_0 >1$. 
Then the  number of factors needed is at least
$ N \ge \left \lceil \frac{|\gamma|}{ L_0-\frac{1}{L_0}} \right \rceil .$
\end{theorem}

Freedman and He proved the same result with lower bound $|\gamma| /\sqrt{L_0^2-1}$, and the improvement above was observed in \cite{GM} for factoring within a special class of bilipschitz maps.

\begin{proof} 
We follow \cite{FH}, where the only adjustment needed is the sharp form of Theorem \ref{thm:pointwise}. For simplicity, we assume $\gamma >0$, the $\gamma <0$ case being similar.
The crucial observation of \cite{FH} is the subadditivity of the rate of rotation under composition. In our notation, for $f$ and $g$ bilipschitz maps,
\[ \overline{\gamma}_{g \circ f}(z) \leqslant \overline{\gamma}_f(z)+\overline{\gamma}_{g}(f(z)), 
\]
provided that $ \overline{\gamma}_{g \circ f}(z)>0$. 
Since $\overline{\gamma}_{s_\gamma}(0)=\gamma$, the repeated application of subadditivity and the estimate of Theorem \ref{thm:pointwise} implies
\[ \gamma \leqslant N  \left(L_0-\frac{1}{L_0} \right),
\]
as required.
\end{proof}

\begin{remark}
The content of Theorem \ref{thm:factoring} is that bilipschitz factoring even when exists might need exponentially more factors than optimal quasiconformal factoring.
One may visualize the difference in this particular example by considering  the maps $ f_\tau(z) = \frac{z}{|z|} |z|^{\alpha(1+i \gamma)}$, where $\tau = \alpha(1+i \gamma)$, with the parameter space 
$\mathbb H=\{ \re \tau >0 \}$. Here the logarithm of the quasiconformal distortion $\log K(f_\tau \circ  f_{\tau'}^{-1})$ equals the hyperbolic distance  $d_{\mathbb H}(\tau, \tau')$. If we now  start at $s_\gamma(z)=z|z|^{i \gamma}$,  travelling to the identity along the hyperbolic geodesic in $\mathbb H$ provides  us the optimal quasiconformal factoring, while bilipschitz factoring requires travelling along the horocycle $\alpha=1$.
\end{remark}


\bibliographystyle{amsplain}

\end{document}